\documentclass[11pt,a4paper,reqno]{amsart}
\usepackage[english]{babel}
\usepackage[applemac]{inputenc}
\usepackage[T1]{fontenc}
\usepackage{palatino}
\usepackage{amsmath}
\usepackage{amssymb}
\usepackage{amsthm}
\usepackage{amsfonts}
\usepackage{graphicx}
\usepackage{color}
\usepackage{vmargin}
\usepackage{mathtools}

\setmarginsrb{20mm}{20mm}{20mm}{20mm}{10mm}{10mm}{10mm}{10mm}

\usepackage[colorlinks = true, citecolor = black]{hyperref}
\newtheorem{theorem}{Theorem}[section]
\newtheorem{obs}{Observation}[section]

\newtheorem{cor}[theorem]{Corollary}
\newtheorem{prop}[theorem]{Proposition}

\newtheorem{lem}[theorem]{Lemma}

\theoremstyle{definition}
\newtheorem{defn}{Definition}[section]

\newtheorem{remark}[theorem]{Remark}

\newtheorem*{theorem*}{Theorem}
\newtheorem{ex}{Example}[section]

\def\vint{\mathop{\mathchoice%
          {\setbox0\hbox{$\displaystyle\intop$}\kern 0.22\wd0%
           \vcenter{\hrule width 0.6\wd0}\kern -0.82\wd0}%
          {\setbox0\hbox{$\textstyle\intop$}\kern 0.2\wd0%
           \vcenter{\hrule width 0.6\wd0}\kern -0.8\wd0}%
          {\setbox0\hbox{$\scriptstyle\intop$}\kern 0.2\wd0%
           \vcenter{\hrule width 0.6\wd0}\kern -0.8\wd0}%
          {\setbox0\hbox{$\scriptscriptstyle\intop$}\kern 0.2\wd0%
           \vcenter{\hrule width 0.6\wd0}\kern -0.8\wd0}}%
          \mathopen{}\int}

\newcommand{\R}{\mathbb{R}}
\newcommand{\Rn}{{\mathbb R}^n}

\newcommand{\ep}{\varepsilon}
\newcommand{\Om}{\Omega}
\newcommand{\om}{\omega}
\newcommand{\ud}{\mathrm {d}}
\newcommand{\Hn}{\mathbb{H}^n}

\newcommand{\Hei}{{\mathbb{H}}^{1}}

\newcommand{\dist}{\operatorname{dist}}
\newcommand{\diam}{\operatorname{diam}}

\newcommand{\A}{\mathcal{A}}

\definecolor{blau}{rgb}{0.1,0.0,0.9}
\definecolor{violet}{rgb}{0.54, 0.17, 0.89}
\newcommand{\blue}{\color{blau}}

\newcommand{\kom}[1]{}
\renewcommand{\kom}[1]{{\bf \blue /#1/}}

\newcounter{komcounter}
\numberwithin{komcounter}{section}

\def\XXint#1#2#3{{\setbox0=\hbox{$#1{#2#3}{\int}$}
		\vcenter{\hbox{$#2#3$}}\kern-.5\wd0}}

\makeatletter
\@namedef{subjclassname@2020}{\textup{2020} Mathematics Subject Classification}
\makeatother

\begin{document}

\title[]{Carleson measures on domains in Heisenberg groups}
\author[T.\ Adamowicz]{Tomasz Adamowicz{\small$^1$}}
\address{The Institute of Mathematics, Polish Academy of Sciences \\ ul. \'Sniadeckich 8, 00-656 Warsaw, Poland}
\email{tadamowi@impan.pl}
\author[M. Grysz\'owka ]{Marcin Grysz\'owka{\small$^1$}}
\address{Institute of Mathematics, Polish Academy of Sciences, ul. \'Sniadeckich 8, 00-656 Warsaw, Poland\/ \and Faculty of Mathematics, Informatics and Mechanics, University of Warsaw, ul. Banacha 2, 02-097 Warsaw, Poland}
\email{mgryszowka@impan.pl}

\thanks{{\small$^1$}T.A. and M.G were supported by the National Science Center, Poland (NCN), UMO-2020/39/O/ST1/00058.}
%\date{\today}

\begin{abstract} We study the Carleson measures on NTA and ADP domains in the Heisenberg groups $\Hn$ and provide two characterizations of such measures: (1) in terms of the level sets of subelliptic harmonic functions and (2) via the $1$-quasiconformal family of mappings on the Kor\'anyi--Reimann unit ball. Moreover, we establish the $L^2$-bounds for the square function $S_{\alpha}$ of a subelliptic harmonic function and the Carleson measure estimates for the BMO boundary data, both on NTA domains in $\Hn$. Finally, we prove a Fatou-type theorem on $(\ep, \delta)$-domains in $\Hn$.

Our work generalizes results by Capogna--Garofalo and Jerison--Kennig.
\end{abstract}

\maketitle

\section{Introduction}
%\komT{M.in. Carleson measures: history and importance (Carleson, Garnett, Varopoulous)}
The Carleson measures play an important role in geometric mapping theory and, especially in recent years, also in the studies of relations between geometry, analysis and the measure theory.  The importance of such measures has been growing in the last decade via the results on PDEs on rough domains, for instance, the studies of the solvability of the Dirichlet problems for elliptic equations, in analysis of the boundary behaviour of harmonic functions, also in relations to the square functions on NTA domains or uniformly rectifiable sets, see e.g.~\cite{hmm, hlm, hmmtz}. From our point of view the two main motivations come from the investigations of the uniform rectifiability and the $\ep$-approximation, see e.g.~\cite{bh, gmt, ht} and from the Hardy spaces of quasiconformal mappings, see~\cite{AK, af}. Moreover, it turns out that the Carleson measures are closely related to the geometry of functions and mappings also in the settings beyond the Euclidean one, for example on homogeneous spaces~\cite{hmmm} and on Riemannian manifolds, see~\cite{mmms, gr} and in the Heisenberg group $\Hei$, see~\cite{af}. Even though, the need for further studies of Carleson measures in the non-Euclidean setting arises, this topic in the subriemannian setting has not yet been explored as much, as in the Euclidean spaces. Therefore, one of the goals of the manuscript is to pursue this direction of investigations. In particular, we focus our attention on the Heisenberg groups $\Hn$, especially on the first Heisenberg group $\Hei$ and on the subelliptic harmonic functions on bounded non-tangentially accessible domains (NTA domains) and on bounded domains admissible for the Dirichlet problem (ADP domains). The fundamental results in the Euclidean setting that have inspired us are discussed in  Chapters I and VI of the book~\cite{ga} and in~\cite{jk}, while the main tools in the potential theory in the Heisenberg groups employed in this work are proven in~\cite{cg98, cgn}.

Let us present and briefly discuss our main results. In Section 3 we show the following characterization of the Carleson measures on ADP domain in $\Hei$ in terms of the level sets of the harmonic functions. The lemma is well known in the setting of the upper-half plane, see Lemma 5.5, Chapter I in \cite{ga}. 

\begin{theorem}\label{thm-main1}
 Let $\Om\subset \Hei$ be a smooth $ADP$ domain with $3$-regular boundary and $\mu$ be a positive measure on $\Om$. Then $\mu$ is a Carleson measure on $\Om$ if and only if there exists a constant $C=C(\alpha)$ such that for every harmonic function $u$ on $\Om$ and every $\lambda>0$ it holds that
 \begin{equation}\label{cond:l55}
 \mu(\{x\in \Om: |u(x)|>\lambda\})\leq C \sigma(\{\om\in \partial \Om: N_{\alpha}u(\om)>\lambda\}),
 \end{equation}
 where $N_{\alpha}u$ stands for the non-tangential maximal function of $u$ (see Definition~\ref{def: Nmax}) and $\sigma$ is the surface measure on $\partial \Om$, i.e. $\sigma=H^2\lfloor \partial \Om$. Moreover, if $C$ is the least constant such that ~\eqref{cond:l55} holds, then the Carleson constant of $\mu$ satisfies $\gamma_{\mu} \approx_{\alpha} C$.
\end{theorem}

While the proof of the sufficiency part of the theorem follows by applying fairly general approach based on the Whitney-type decomposition, the proof of the necessity part relies on the potential-theoretic properties of harmonic functions, including the boundary Harnack estimate in~\cite{gp} and the results proven in~\cite{cgn}. 
%Let us remark, that the original proof in~\cite{ga} based on the direct estimates for the Poisson extension of the piecewise constant function would be difficult to repeat due to 

Our next result generalizes a characterization of Carleson measures on the unit disc in the Euclidean plane, cf. Lemma 3.3  in Chapter VI.3 in \cite{ga}, see Section 4 for the detailed discussion. One of the key features that give the result in the plane is the rich family of M\"obius self-transformations of a disc, a property which is no longer true in the subriemmanian setting due to the rigidity of Carnot groups. However, recently in~\cite[Section 4.1]{af} a counterpart of M\"obius self-maps of a ball in $\Rn$ has been introduced on the Kor\'anyi--Reimann unit ball $B(0,1)\subset \Hei$ by the price of giving up that the target domain remains a ball, see the definition of maps $T:=T_{x,a,\rho}$ in~\eqref{def:Tmob} and their property~\eqref{incl:balls}. The following result characterizes the Carleson measures on $B(0,1)$ in terms of the boundary growth of $1$-quasiconformal mappings $T$.

\begin{theorem}\label{thm-main3}
 A measure $\mu$ on the Kor\'anyi--Reimann unit ball $B:=B(0,1)\subset \Hei$ is a Carleson measure if and only if
 \begin{equation}%\label{cond: thm-m3}
  \int_{B} \left( \frac{d(T_{x,a,\rho}(y),\partial T_{x,a,\rho} (B))}{d(y,\partial B)}\right)^3 \ud \mu(y)=M<\infty,
 \end{equation}
 for all $x\in B$, $a\in \mathbb{H}^1 \setminus \overline{B}$, and $\rho>0$ such that 
$\rho \lesssim \min\{d(x,\partial B), d(a,\partial B) \}$ and $\rho \approx d(a,x)$.
\end{theorem}

In Remark 4.1 we also point to the generalization of the above theorem to the setting of higher order Heisenberg groups $\Hn$ for $n\geq 2$.

Section 5 contains main results of our work. The first one is the $L^2$-estimate for the square function of a subelliptic harmonic function on a bounded NTA domain in $\Hn$ with respect to the $L^2$ boundary data and the harmonic measure $\om$. The result generalizes Theorem 9.1 in~\cite{jk} proved for bounded NTA domains in $\Rn$.

\begin{theorem}[$L^2$-boundedness of the square function]\label{thm-jk91}
 Let $\Om\subset \Hn$ be a bounded NTA domain. Let further $f\in L^2(\ud \om)$ and $u(x):=\int_{\partial \Om} f(y)\ud \om^x(y)$. Then, the following estimate holds for the square function $S_{\alpha}$ of a harmonic function $u$ in $\Om$
 \[
  \|S_{\alpha}u\|_{L^2(\ud \om)}\leq C\|f\|_{L^2(\ud \om)}.
 \]
\end{theorem}

Our second main result is the subriemannian counterpart of the Euclidean result, i.e. Theorem  9.6 in~\cite{jk}. Moreover, it also generalizes Theorem 3.4 in~\cite[Chapter VI.3 ]{ga} for the unit disc in the plane, see Remark~\ref{rem-Garnett-34}. We further refer to Example~\ref{ex-ball} for the case of the unit gauge ball in $\Hn$, where the Green function $G$ in the assertion of Theorem~\ref{thm-jk96} can be explicitly estimated from below in terms of the distance function, thus providing more classical and handy estimate~\eqref{est-Carl-true}. In order to obtain this estimate we prove Proposition A.1 in the Appendix.

\begin{theorem}[Carleson measure estimate]\label{thm-jk96}
 Let $\Om\subset \Hn$ be a bounded NTA domain and $u$ be a subelliptic harmonic in $\Om$ such that $u(x)=\int_{\partial \Om} f(y)\ud \om^x(y)$ for some $f\in BMO(\partial \Om)$. Then, for any ball $B(x_0, r)$ centered at $x_0\in \partial \Om\setminus \Sigma_{\Om}$ with any $0<r<r_0\leq \min\{1, \frac{d(x_0,\Sigma_{\Om})}{M}\}$ it holds that
 \[
 \int_{B(x_0,r)\cap \Om} |\nabla_H u|^2 G(x, A_r(x_0))\ud x\leq C\om(B(x_0,r)\cap \partial \Om),
 \]
 where constant $C$ depends on $n, M, r_0$ and $\|f\|_{BMO(\partial \Om)}$. 
\end{theorem}

Among corollaries of Theorem~\ref{thm-jk96} we show the corresponding Carleson estimates on an ADP domain (Corollary 5.1) and on the (Euclidean) $C^{1,1}$-domain (Corollary 5.2).

The proof of Theorem~\ref{thm-jk96} consists of several steps and auxiliary observations which largely follow the steps of the corresponding proof of Theorem  9.6 in~\cite{jk}. However, we expand several arguments in~\cite{jk} and clarified steps which in the new setting of Heisenberg groups require using the subriemannian tools.

Our last result, proved in Section 6, is a counterpart of the classical Fatou theorem for harmonic functions on $(\ep,\delta)$-domains in $\Hn$, under the condition of the $L^p$-integrability of the gradient of the function. The $(\ep,\delta)$-domains in $\Hn$ can be thought of as the quantified version of the uniform domains and contain large family of NTA domains, see the detailed presentation in Section 6.

\begin{theorem}\label{thm-Fatou}
 Let $\Om\subset \Hn$ be a bounded $(\ep, \delta)$-domain and let further $u$ be harmonic in $\Om$. If $\int_{\Om} |\nabla_H u|^p<\infty$ for some $1<p\leq 2n+2$, then $u$ has nontangential limits on $\partial \Om$ along horizontal curves in $\Om$ outside the set of $p$-Sobolev capacity zero.
\end{theorem}

This result extends previous observations in the Heisenberg setting in two ways:
\smallskip
\\
\noindent (1) the considered domains are slightly more general than in a Fatou theorem on NTA domains in $\Hn$ (\cite[Theorem 4]{cg98}) and in $\Rn$ (\cite[Theorem 6.4]{jk}); 
\smallskip

\noindent (2) the assertion gives the existence of nontangential limits not only up to the measure zero set as e.g. in~\cite{cg98}, but outside the set of $p$-Sobolev capacity zero, which is a refined measure.

\section{Preliminaries}

In this section we recall key definitions employed in the paper. Our presentation includes the Heisenberg group, various types of domains and their geometry, basic information on subelliptic harmonic functions and Green functions in subriemannian setting, the Carleson measures, the non-tangential maximal function and the BMO spaces.
 
\subsection{Heisenberg groups} The $n$-th Heisenberg group $\Hn$ as a set is $\mathbb{R}^{2n}\times\mathbb{R}\simeq\mathbb{C}^n\times\mathbb{R}$ with the group law given by 
\[
(z_1,\dots,z_n,t)\cdot(z_1',\dots, z_n',t')=\Big(z_1+z_1',\dots, z_n+z_n',t+t'+2{\rm Im}\Big(\sum_{i=1}^{n}z_i\overline{z_i'}\Big)\Big),
\]
where $(z_1,\dots,z_n,t)=(x_1,y_1,\dots,x_n,y_n,t)$.  Furthermore, we define the following left-invariant vector fields 
\begin{align*}
X_i(p)&=\frac{\partial}{\partial x_i}+2y_i\frac{\partial}{\partial t}, \hspace{5mm} Y_i(p)=\frac{\partial}{\partial y_i}-2x_i\frac{\partial}{\partial t},\quad i=1,\ldots, n, \hspace{5mm} T=\frac{\partial}{\partial t}
\end{align*}
for which the only nontrivial brackets are 
$$
[X_i, Y_i] =-4T \quad i=1,\dots,n.
$$ 
The horizontal space at $p\in \Hn$ is given pointwise by
\begin{align*}
\mathcal{H}_p\Hn=\textnormal{span}\{X_1(p),Y_1(p),\dots,X_n(p),Y_n(p)\}.
\end{align*}
Let $\gamma:[0,S]\rightarrow\mathbb{R}^{2n+1}$ be an absolutely continuous curve. We say that $\gamma$ is horizontal if ${\dot{\gamma}(s)\in\mathcal{H}_{\gamma(s)}\Hn}$ for almost every $s$. Now, we equip $\mathcal{H}_p\Hn$ with the left invariant Riemannian metric such that vector fields $X_i,Y_i$ are orthonormal, and so if $v\in\mathcal{H}_p\Hn$ is given as $v=\sum_{i=1}^n a_i X_i(p)+b_i Y_i(p)$, then the following expression defines a norm $|v|_H=\sqrt{\sum_{i=1}^n (a_i^2+b_i^2)}$. In a consequence, we define the Carnot-Carath\'{e}odory distance in $\Hn$ as follows:
\[
d_{CC}(p,q)=\inf_{\Gamma_{p,q}}\int_a^b |\dot{\gamma}(s)|_H\ud s,
\]
where $\Gamma_{p,q}$ denotes the set of all horizontal curves joining $p$ and $q$, such that $\gamma$ joins points $p$ and $q$: $\gamma(a)=p$ and $\gamma(b)=q$.

 Equipped with the above structure, the Heisenberg group $\Hn$ becomes a subriemannian manifold and a Carnot--Carath\'eodory group, in addition to being a metric space. However, the Carnot-Cara\-th\'{e}\-odory distance can be troublesome and, hence, we introduce the so-called Kor\'{a}nyi--Reimann distance, defined as follows:
\[
d_{\Hn}(p,q)=\|q^{-1}\cdot p\|,
\]
where the pseudonorm is given by
\[
\|p\|:=\|(z,t)\|:=\left(|z|^4+t^2\right)^{\frac{1}{4}}.
\] 
The Kor\'{a}nyi--Reimann distance is equivalent (comparable) to $d_{CC}$ and hence both distances generate the same topology, see e.g.~\cite{bel}. However, $d_{\Hn}$ is easier in computations and therefore, throughout this paper we use Kor\'{a}nyi--Reimann distance $d_{\Hn}$, rather then the subriemannian distance.

Finally, we recall that the left invariant Haar measure on $\Hn$ is simply the $(2n+1)$-dimensional Lebesgue
measure on $\Hn$ and it follows that $\Hn$ is $Q$-Ahlfors regular, with $Q=2n+2$, i.e.
there exists a positive constant $c$ such that for all balls $B$ with radius $r>0$ we have
\[
\frac{1}{c} r^Q \leq \mathcal{H}^Q(B) \leq c r^Q,
\]
where $\mathcal{H}^Q$ stands for the $Q$-dimensional Hausdorff measure induced by $d_{\Hn}$.

%and it turns out that the balls in Kor\'{a}nyi distance have some properties which balls in $d_{CC}$ do not satisfy. We will explain what we mean by it in the next section.

\subsection{Geometry of domains} One of the fundamental types of domains studied in our work are the NTA domains and the ADP domains, whose definitions and basic properties we now recall.

%Recall the following Definition 5.11 in~\cite{cgn}.
 
\begin{defn}[NTA domain, cf. Definition 5.11 in~\cite{cgn}]\label{def:NTA}
We say that a domain $\Om\subset \Hn$ is a nontangentially accessible domain (NTA, for short) if there exist constants $M, r_0>0$ such that:
\begin{enumerate}
\item (Interior corkscrew condition). For any $x\in\partial\Om$ and $r\le r_0$ there exists $A_r(x)\in\Om$ such that 
\[
\frac{r}{M}<d_{CC}(A_r(x),x)\le r\quad \hbox{and} \quad d_{CC}(A_r(x),\partial\Om)>\frac{r}{M}.
\]
\item (Exterior corkscrew condition). The complement $\Om^c:=\Rn\setminus \Om$ satisfies the interior corkscrew condition.
\item (Harnack chain condition). For every $\varepsilon>0$ and $x,y\in\Om$ such that $d_{CC}(x,\partial\Om)>\varepsilon$, $d_{CC}(y,\partial\Om)>\varepsilon$ and $d_{CC}(x,y)<C\varepsilon$ there exists a sequence of balls $B_1,\dots, B_k$ with the following properties:
\begin{enumerate}
\item $x\in B_1$ and $y\in B_k$,
\item $\frac{r}{M}<d_{CC}(B_i(x,r),\partial\Om)<Mr$ for every $i=1,\dots,k$,
\item $B_i\cap B_{i+1}\neq\emptyset$ for $i=1,\dots,k-1$,
\item length of the chain $k$ depends on $C$ but not on $\varepsilon$.
\end{enumerate}
\end{enumerate}
\end{defn}

In the corresponding Definition 1 in~\cite{cg98} the analogous notion of the $X$-NTA domains is considered. There, one
lets $X=\{X_1,\dots,X_m\}$ be a family of smooth vector fields satisfying the H\"ormander rank condition, and so $d_{CC}$ denotes the Carnot-Carath\'{e}odory distance related to $X$. For example, in the Heisenberg group $\Hn$ our family of vector fields is $X:=\{X_1,Y_1,\dots,X_n,Y_n\}$, cf. Section 2.1 above.

The notion of the NTA domain originates from a work of Jerison--Kennig, see~\cite[Section 3]{jk}. Notice that the above definition makes sense also in the setting of metric space, in which case, the distance need not be induced by a family of vector fields. 

Examples of NTA domains in $\Rn$ encompass Lipschitz domains, Zygmund domains and quasispheres (snow-flake domains).  Another example of an NTA domain is a complement of a planar Cantor set in a large enough ball, denoted by $\Om$. It turns out that $\Om$ satisfies our definition, even though such a Cantor set is not rectifiable as a part of the $1$-dimensional boundary of $\Om$. Intuitively speaking, one can think that conditions (1) and (2) exclude both interior and exterior cusps, while condition (3) eliminates a possibility of slits within a domain or narrowings that are infinitely thin. Examples of NTA domains in $\Hn$, or in more general Carnot groups, include:
\smallskip
\\
\indent - bounded $C^{1,1,}$ sets with cylindrical symmetry (Theorem 5 in~\cite{cg98}),

- level sets of fundamental solutions of the real part of the sub-Laplacian (Corollary 2 in~\cite{cg98}), 

- balls in the metric $d_{\Hn}$, see Corollary 4 and Proposition 1 in~\cite{cg98},

- an image of an NTA domain in $\Hn$ under the global quasiconformal map $f:\Hn \to \Hn$ is an NTA domain, see~\cite{ct}.

We refer to Section 5 in~\cite{cg98} for further examples of NTA domains. However, it turns out that balls in $d_{CC}$ are not NTA domains.
This partially motivates that from the point of view of our studies, the $d_{\Hn}$ distance has an advantage over Carnot-Carath\'e\-odory distance. 
\smallskip
\\
From now on, unless specified differently, let us denote by $d:=d_{\Hn}$.
\smallskip  
% Since the distances $d_{CC}$ and $d_{\Hn}$ are equivalent on $\mathbb{H}^n$ we can replace the distance $d_{CC}$ in the definition of $X$-NTA sets with $d_{\Hn}$. Therefore, we can forget about vector fields and the distance $d_{CC}$ and use only $d_{\Hn}$ for dealing with $X$-NTA domains that from now on we will call NTA domains as we no longer care about family $X$. 
 
Basing on the notion of the NTA domains we now recall the second fundamental type of domains considered in this work, namely the so-called domains \emph{admissible for the Dirichlet problem}, ADP for short, see~\cite{cgn}. Such a class is defined by combining the above notion of NTA domains with the existence of a uniform outer ball. As observed in~\cite{cgn} , it can be viewed as the closest nonabelian counterpart of the class of $C^{1,1}$ domains from Euclidean analysis.

\begin{defn}[cf. Definition 2.1 in~\cite{cgn}]\label{def: ADP}
 We say that a bounded domain $\Om \subset \Hei$ is \emph{admissible for the Dirichlet problem}, denoted by ADP, if $\Om$ is NTA and satisfies the uniform outer ball condition with respect to the metric $d$. 
% \komT{$N$ w CGN jest troche inne niz nasze, czy definicje sa porownywalne? Raczej tak.}\komM{W CGN funkcja $N$ rozni sie tylko stala rowna 16. Zatem metryki sa porownywalne. W innych dowodach w CG nie bylo istotne, jaka jest stala, wiec wyniki z CG sa prawdziwe zarowno dla stalej rownej 1, jak i 16. Przejrzalem tez dowody w CGN i przez stala 16 w dowodach w rozdziale 2. wyskakuja jakies stale, ale nie widze problemu, zeby nie mogla byc stala 1.}
\end{defn}

\subsection{Subelliptic harmonic functions and Green functions}

Below we collect some of the basic definitions and potential theoretic results for the theory of subelliptic harmonic functions in Heisenberg groups $\Hn$.
%\begin{defn}%\label{def: grad}

Let $\Om\subset \Hn$ be an open set in the Heisenberg group $\Hn$. We say that a function $u:\Om\rightarrow\mathbb{R}$ belongs to the \emph{horizontal Sobolev space} $HW^{1,2}(\Om)$, if $u\in L^2(\Om)$ and the horizontal derivatives $X_i u, Y_i u$ for $i=1,\ldots, n$ exist in the distributional sense and belong to $L^2(\Om)$. Similarly, we define the local horizontal Sobolev space $HW_{loc}^{1,2}(\Om)$.

The \emph{horizontal gradient} $\nabla_H u$ is given by the following equation:
\[
\nabla_H u:=\sum_{i=1}^n (X_iu)X_i+(Y_iu)Y_i.
\]
Next, we define the sub-Laplace operator of $u$: 
\[
\Delta_Hu:=\sum_{i=1}^n (X_i)^2u+(Y_i)^2u,
\]
and say that $u\in HW_{loc}^{1,2}(\Om)$ is \emph{subelliptic harmonic} in $\Om$, if $\Delta_Hu=0$ in the weak sense. In what follows, for the sake of simplicity, we will omit the word subelliptic and write \emph{harmonic functions}, instead. 

Recall that harmonic functions in $\Hn$ are smooth (in fact analytic) and satisfy the weak maximum principle and the Harnack inequality, see Chapters 8 and 5 in~\cite{blu}, respectively.
%\end{defn}

Let $G(x,y)=G(y,x)=G_{\Om}(x,y)$ denote the Green function for the sub-Laplacian and for the domain $\Om\subset \Hn$. We refer to~\cite{cg98} and to Chapter 9 in \cite{blu} for definitions and basic properties of Green functions. Moreover, in the Appendix we provide a proof of one of the standard properties of Green functions needed in Example~\ref{ex-ball}. The result is likely a mathematical folklore in $\Hn$, but since we did not find it explicitly in the literature for Carnot groups, we provide the full argument.

The following observations from~\cite{cg98} will frequently be used, especially in Section 5. Here, we formulate them for gauge balls rather then for the metric balls. This is justified by the equivalence of both metrics in $\Hn$.

Let $\Delta(x,r):=B(x,r)\cap \partial \Om$ denote the surface ball at $x\in \partial \Om$ with radius $r>0$.

\begin{theorem}[Dahlberg-type estimate, cf.~Theorem 1~\cite{cg98}]\label{cg-thm21}
Let $\Om\subset \Hn$ be an NTA domain with parameters $M, r_0>0$ and let further $x_0\in\partial\Om$ and $r < \frac{r_0}{2}$. Then, there exist $a >1$ and $C >0$, depending on $\Delta_H, M$ and $r_0$, such that 
for every $x\in\Om\setminus B_d(x_0, ar)$
\[ 
C \frac{|B_d(x,r)|}{r^2}G(x, A_r(x_0)) \le \om^x (\Delta(x_0, r)) \le C^{-1} \frac{|B_d(x,r)|}{r^2} G(x, A_r(x_0)),
\]
where $G$ denotes a Green function of $\Om$.
\end{theorem}
\begin{theorem}[Local comparison theorem, cf.~Theorem 3~\cite{cg98}]\label{cg-thm22}
Let $\Om\subset \Hn$ be an NTA domain with parameters $M, r_0>0$ and let further $x_0\in\partial\Om$ and $0< r < \frac{r_0}{M}$. If $u,v$ are nonnegative harmonic functions in $\Om$, that continuously vanish on $\Delta(x_0,Mr)$, then for any $x\in B_d(x_0,\frac{r}{2M})\cap\Om$ one has  
\[ 
\frac{u(x)}{v(x)} \le C\frac{u(A_r(x_0))}{v(A_r(x_0))} ,
\]
for some constant $C > 0$ which depends only on $\Delta_H, M$ and $r_0$.
\end{theorem}

\subsection{Carleson measures and related notions in Harmonic analysis}

Recall that a non-empty open connected set $\Om\subset X$ of a metric space $(X,d)$ has \emph{$s$-regular boundary} for some $s>0$, if there exists a constant $C\geq 1$ such that
\begin{displaymath}
\frac{1}{C}\, r^s \leq \mathcal{H}^s(B(x,r)\cap \partial \Omega)\leq C \,r^s,\quad \text{for all }x\in \partial \Omega\text{ and
}0<r<\mathrm{diam}(\partial \Omega).
\end{displaymath}

Next, we recall the definition of the Carleson measure, here formulated in the setting of the Heisenberg group $\Hei$.

\begin{defn}[Carleson measure]\label{def: Carleson-m}
Let $1\leq \alpha<\infty$ and $s>0$.  We say that a positive Borel measure $\mu$ on an open connected set $\Om\subset \Hei$ with non-empty $s$-regular boundary is an \emph{$\alpha$-Carleson measure on $\Om$}, if there exists a
constant $C>0$ such that
\begin{equation}\label{eq:CarlesonCOnst}
\mu(\Om \cap B(x,r))\leq C r^{\alpha s},\quad \text{for all }x\in \partial \Om\text{ and }r>0.
\end{equation}
The $\alpha$-\emph{Carleson measure constant} of $\mu$ is defined
by
\begin{displaymath}
\gamma_{\alpha}(\mu):= \inf\{C>0\text{ such that \eqref{eq:CarlesonCOnst} holds for all $x \in\partial \Om$ and
$r>0$}\}
\end{displaymath}
We also call $1$-Carleson measures simply \emph{Carleson measures}.
\end{defn}

We define two objects that are essential for analyzing the boundary behavior of harmonic functions. 

\begin{defn}\label{def: Nmax}
Let $u:\Om\rightarrow\mathbb{R}$ be a continuous function. We define a nontangential maximal function $N_{\alpha}u:\partial\Om\rightarrow\mathbb{R}$ as follows:
\[
(N_{\alpha}u)(x)=\sup\{|u(y)|:y\in\Gamma_{\alpha}(x)\},
\]
where $\Gamma_{\alpha}(x)=\{y\in\Om: d(y,x)<(1+\alpha)d(y,\partial\Om)\}$ is a cone with vertex $x\in\partial\Om$ and aperture given by $\alpha$.
\end{defn}

In the next definition we assume that the function is $C^1$, but the Sobolev regularity $HW^{1,2}_{loc}$ would suffice as well, cf.~\cite{gmt} for the Euclidean setting. Since the definition below is applied only to harmonic functions on $\Hn$, which are analytic, our regularity assumption is enough.
\begin{defn}\label{def: square}
Let $u:\Om\rightarrow\mathbb{R}$ be a $C^1(\Om)$ function. We define a square function $(S_{\alpha}u)^2:\partial\Om\rightarrow\mathbb{R}$ as follows:
\[
(S_{\alpha}u)^2(x)=\int_{\Gamma_{\alpha}(x)}|\nabla_Hu(y)|^2 d(y,\partial\Om)^{2-Q}\ud y,
\]
where $Q=2n+2$ is a homogeneous dimension of $\Hn$.
\end{defn}

For a given domain $\Om\subset \Hn$ choose a point $y\in \Om$ and consider the harmonic measures $\om^{y}$ on $\Om$. Then for a given $x\in \partial \Om$ and $r>0$ we let $\Delta(x,r):=B(x,r)\cap \partial \Om$ and recall the mean-value of a function $f:\partial \Om\to\R$ on $\Delta(x,r)$:
\[
f_{\Delta(x,r)}:=\vint_{\Delta(x,r)} f(z)\,\ud \om^y(z).
\]

\begin{defn}[Boundary BMO space]\label{defn-bmo-bd}
Let $\Om\subset \Hn$ be a domain. We say that a function $f:\partial \Om \to \R$ belongs to the space $BMO(\partial\Om,\ud\om)$ with respect to the harmonic measure $\om$ in $\Om$, if 
\[
\sup_{\Delta(x,r)}\frac{1}{\om(\Delta(x,r))}\int_{\Delta(x,r)}|f(y)-f_{\Delta(x,r)}|\ud\om<\infty.
\]
\end{defn}
When discussing the NTA domains in $\Hn$ we may omit the reference point in the harmonic measure $\ud\om^z$ and write $\ud\om$ for simplicity, as the following observation asserts that if a function belongs to the space $BMO(\partial\Om,\ud\om^{z_0})$ for $z_0\in\Om$, then it belongs to every space $BMO(\partial\Om,\ud\om^z)$ for all $z\in\Om$. 

\begin{obs} \label{obs-bmo}
%In the definition of space $BMO(\partial\Om,\ud\om)$ we do not use any particular point $z$ to denote $\om=\om^z$.
Let $\Om\subset \Hn$ be a domain. Then, it holds that 
\[
BMO(\partial\Om,\ud\om^z)=BMO(\partial\Om,\ud\om^{z_0}),
\]
for any points $z,z_0\in\Om$.
\end{obs}
\begin{proof}
Suppose that $f\in BMO(\partial\Om,\ud\om^{z_0})$ and denote by $K(z,y)$ the Radon-Nikodym derivative $\frac{\ud\om^z}{\ud\om^{z_0}}(y)$, see Section 4 in \cite{cg98}. By the second remark in~\cite[Section 4]{cg98}, we know that for a fixed $z$, a function $y\mapsto K(z,y)$ is continuous and hence bounded by a constant $C(z,z_0)$ as $\partial\Om$ is compact. Therefore
\begin{align*}
\sup_{\Delta(x,r)}\frac{1}{\om^{z}(\Delta(x,r))}\int_{\Delta(x,r)}&|f(y)-f_{\Delta(x,r)}|\ud \om^{z}(y)\\
&=\sup_{\Delta(x,r)}\frac{1}{\om^{z}(\Delta(x,r))}\int_{\Delta(x,r)}|f(y)-f_{\Delta(x,r)}|K(z,y)\ud \om^{z_0}(y)\\
&=\sup_{\Delta(x,r)}\frac{\om^{z_0}(\Delta(x,r))}{\om^{z}(\Delta(x,r))}\frac{1}{\om^{z_0}(\Delta(x,r))}\int_{\Delta(x,r)}|f(y)-f_{\Delta(x,r)}|K(z,y)\ud \om^{z_0}(y)\\
&\lesssim_{z,z_0} \sup_{\Delta(x,r)}\frac{C(z,z_0)}{\om^{z_0}(\Delta(x,r))}\int_{\Delta(x,r)}|f(y)-f_{\Delta(x,r)}|\ud \om^{z_0}(y)<\infty
\end{align*}
as a quotient $\frac{\om^{z_0}(\Delta(x,r))}{\om^{z}(\Delta(x,r))}$ is bounded by a constant depending on the Harnack inequality constant, as well as on points $z_0$ and $z$.  Hence $f\in BMO(\partial\Om,\ud\om^{z})$.
\end{proof}

We remark that on good domains, for example on NTA domains, the dependence on $z$ and $z_0$ discussed in the end of the above proof, can be reduced to the dependence on the diameter of the domain and the constants $r_0$ and $M$ in Definition~\ref{def:NTA}.

\section{Characterizations of Carleson measures on ADP-domains}

The purpose of this section is to show Theorem~\ref{thm-main1}, which can be understood as the nonabelian counterpart of the well-known characterization of the Carleson measures in the upper-half plane $\R^2_{+}$, see Lemma 5.5 in~\cite[Section 5, Ch. I]{ga}. However, here we prove it only for bounded domains.

%\begin{theorem}[cf. Lemma 5.5, Chapter I in \cite{ga} ]\label{thm-main1}
% Let $\Om\subset \Hei$ be a smooth $ADP$ domain with $3$-regular boundary and $\mu$ be a positive measure on $\Om$. Then $\mu$ is a Carleson measure on $\Om$ if and only if there exists a constant $C=C(\alpha)$ such that for every harmonic function $u$ on $\Om$ and every $\lambda>0$ it holds that
% \begin{equation}\label{cond:l55}
% \mu(\{x\in \Om: |u(x)|>\lambda\})\leq C \sigma(\{\om\in \partial \Om: N_{\alpha}u(\om)>\lambda\}),
% \end{equation}
% where $\sigma$ is the surface measure on $\partial \Om$, i.e. $\sigma=H^2\lfloor \partial \Om$. Moreover, if $C$ is the least constant such that ~\eqref{cond:l55} holds, then the Carleson constant of $\mu$ satisfies $\gamma_{\mu} \approx_{\alpha} C$.
%\end{theorem}

\begin{remark}
Upon the necessary modifications, Theorem~\ref{thm-main1} can be as well formulated for the smooth ADP domains in $\Hn$ for $n\geq 1$. However, for the sake of the simplicity of the presentation and in order to emphasize the similarity to the corresponding result in~\cite{ga}, we restrict our discussion to $\Hei$ only.
\end{remark}

The proof of the sufficiency part relies on the corresponding one for Proposition 6.3 in~\cite{af} and in fact holds for continuous functions in doubling metric spaces.

In order to show the necessity part of the assertion we adapt the idea of the proof of Lemma 5.5 in~\cite[Section 5, Ch. I]{ga} for the Carleson measures on the upper half plane $\R^2_{+}$ and the Euclidean harmonic functions. There, by choosing the constant boundary data $4\lambda$ with support contained in the interval $I\subset \R$ and by defining the harmonic function $u$ as the convolution of the Poisson kernel in the upper half plane $\R^2_{+}$ with the function $4\lambda \chi_I$, one shows that the superlevel set $\{x\in \R^2_{+}: u(x)>\lambda\}$ contains the square $Q$ with base $I$ and so its measure satisfies: $\mu(Q)\leq \mu(\{x\in \R^2_{+}: u(x)>\lambda\})$. This combined with the weak-$L^1$ estimate for the Hardy--Littlewood maximal function gives the assertion of the theorem. 

Our strategy of the proof relies on the following facts: first, existence of the Poisson kernel on $ADP$ domains allows us to construct the appropriate harmonic function $u$. Then we invoke the harmonic measure representation of $u$ together with the mutual absolute continuity of the harmonic measure with respect to the surface measure. Finally, the subelliptic counterparts of the weak-$L^1$ estimates and the estimates for the non-tangential maximal function allow us to conclude the necessity part of the proof.

\begin{defn}\label{defn:char}
Let $\Om\subset \Hn$ be a domain. We say that a point $x\in \partial \Om$ is characteristic if the tangent space to $\partial \Om$ at $x$ is horizontal. The set of all such points in $\partial \Om$ is denoted by $\Sigma_{\Om}$. 
\end{defn}

%%\komT{Czy umiemy pokazac Thm. 3.2? Patrz komentarz 3.7}
%%The case of an unbounded domain $\Heip$ requires a separate proof, see Theorem~\ref{thm-main2} below. There, we show the existence of the Green function and the related Poisson kernel for $\Heip$ and construct the harmonic function whose superlevel set contains a half-ball. \komT{Let us mention that according to our best knowledge, both the Green function and the corresponding Poisson kernel are new in the literature, see~\cite{grs, cgn}. Nie uzywamy znalezionych wzorow explicite, wiec to zdanie raczej do modyfikacji.}

For the readers convenience we will now recall results from~\cite{cgn} that are essential for the proof of Theorem~\ref{thm-main1}:
\begin{itemize}
\item[(1)] (Theorem 1.1 \cite{cgn}). Let $\Om\subset\mathbb{H}^n$ be a smooth ADP domain. Then, for any $x\in \Om$ the (subelliptic) harmonic measure $\ud\om^{x}$ and the surface measure $\ud \sigma$ are mutually absolutely continuous. Moreover, for every $p >1$ it holds that $L^p(\partial \Om, \ud \sigma)\subset L^1(\partial \Om,\ud\om^{x})$.

\item[(2)] (Theorem 5.5 \cite{cgn}). Let $\Om\subset\mathbb{H}^n$ be an NTA domain. Fix $x_0\in \Om$ and for a given $\phi\in L^1(\partial \Om,\ud\om^{x_0})$  define the following function
\[
 u(x):=\int_{\partial\Om}\phi(y) \ud \om^{x}(y),\quad x\in \Om.
\]
Then $u$ is subelliptic harmonic in $\Om$ and the following estimate holds for the non-tangential maximal function of $u$ and the Hardy--Littlewood maximal operator:
\[
(N_{\alpha}(u))(x)\leq C M_{\om}(\phi)(x),\quad x\in \partial \Om.
\]
\item[(3)] (Theorem 4.9 in~\cite{cgn}). Let $\Om\subset \Hn$ be a smooth domain, then $\sigma(\Sigma_{\Om})=0$.
 \end{itemize}

\begin{proof}[Proof of Theorem~\ref{thm-main1}]
 The sufficiency part of the proof follows from the discussion analogous to the one in the proof of Proposition 6.3 in~\cite{af}. In particular, formula (6.5) in~\cite{af} for $\alpha=1$ and $s=2$ gives assertion~\eqref{cond:l55}. For the sake of completeness of the presentation we now provide some key steps of the reasoning in~\cite{af}. Moreover, for the sufficiency part it is enough that function $u$ in~\eqref{cond:l55} is a continuous function. 
 
 Let $\mu$ be a Carleson measure on $\Om$. We define the following superlevel sets
\begin{displaymath}
E(\lambda):=\{x\in \Omega:\, |u(x)|>\lambda\}\quad\text{and}\quad
U(\lambda):= \{\om \in \partial \Omega:\, N_{\alpha}u(\om)>\lambda\},\quad \lambda >0.
\end{displaymath}
In this notation, assertion~\eqref{cond:l55} reads 
\begin{equation}\label{eq:superlevel}
\mu(E(\lambda)) \leq C \mathcal{H}^2(U(\lambda))\quad\text{for all }\lambda>0.
\end{equation}

As in~\cite{af} we employ the Whitney-type decomposition of $U(\lambda)$ based on the general result \cite[Proposition 4.1.15]{hkst} applied to the metric space $(\partial \Omega, d|_{\partial \Omega})$ and the open set $U(\lambda)$. It allows is to find a countable collection
$\mathcal{W}_{\lambda}=\{B(\omega_i,r_i):\, i=1,2,\ldots\}$ of
balls with $\omega_i \in U(\lambda)$ such that
\begin{equation}\label{eq:Whitney1} U(\lambda) = \bigcup_{i=1,2,\ldots}
B(\omega_i,r_i) \cap \partial \Omega,
\end{equation}
\begin{equation}\label{eq:Whitney2}
\sum_i \chi_{B(\omega_i,2 r_i)\cap \partial \Omega}\leq 2 N^5,
\end{equation}
where $r_i = (1/8) d(\omega_i,\partial \Omega\setminus U(\lambda))$ and $N$ depends only on the $2$-regularity
constant of $\partial \Omega$. 
%To prove the superlevel set estimate, we want to show that $E(\lambda)$ is included in the union of the balls $B(\omega_i,Cr_i)$, for a suitable geometric constant $C=C(\kappa)$. 
Let $x$ be an arbitrary point in $E(\lambda)$, then
\begin{displaymath}
N_{\alpha}u(\omega)>\lambda,\quad \text{for all }\omega\in S(x)=
B\left(x,(1+\alpha)d(x,\partial \Omega)\right)\cap
\partial \Omega,
\end{displaymath}
and hence
\begin{equation*}%\label{eq: S(q)_in_U(lambda)}
S(x) \subset U(\lambda) \overset{\eqref{eq:Whitney1} }{=}
\bigcup_i B(\omega_i,r_i) \cap \partial \Omega,\quad \text{for all
}x\in E(\lambda).
\end{equation*}
Next, for $x\in E(\lambda)$, let $\omega_x \in \partial B$ be such that $d(x,\omega_x)= d(x,\partial \Omega)$, due to compactness of $\partial \Omega$. Thus  $\omega_x\in S(x)$ and, since $S(x) \subset U(\lambda)$, we
moreover know that $d(\omega_x,\partial \Omega \setminus U(\lambda)) \geq d(\omega_x,\partial \Omega \setminus S(x))$.

Furthermore, there exists $i_x\in \{1,2,\,\ldots\}$ such that $\omega_x \in B(\omega_{i_x}, r_{i_x})$. By repeating the reasoning in~\cite{af} we find that
\begin{equation*}%\label{eq:Carleson_4star}
x\in B\left(\omega_{i_x},\left(\tfrac{9}{\alpha}+1\right) r_{i_x} \right),\quad r_{i_x} = \frac{1}{8} d(\omega_{i_x},\partial \Omega\setminus U(\lambda)).
\end{equation*}

%Combining this information, we find that
%\begin{align}\label{eq:triangle}
%r_{i_x} = \frac{1}{8} d(\omega_{i_x},\partial \Omega\setminus
%U(\lambda)) &\geq \frac{1}{8} \left[d(\omega_x,\partial \Omega
%\setminus U(\lambda)) - d(\omega_x,\omega_{i_x}) \right]\notag\\
%&\overset{\eqref{eq:Carleson_star}}{\geq} \frac{1}{8}
%d(\omega_x,\partial \Omega\setminus S(x)) - \frac{1}{8} r_{i_x}.
%\end{align}
%Since $d(\omega_x,x)= d(x,\partial \Omega)$, it is easy to see
%that
%\begin{displaymath}
%B\left(\omega_x,\kappa d(x,\partial \Om)\right)\cap \partial \Omega
%\subset S(x).
%\end{displaymath}
%Hence  the above chain of inequalities implies that
%\begin{displaymath}
%9 \,r_{i_x} \overset{\eqref{eq:triangle}}{\geq}d(\omega_x,
%\partial \Omega \setminus S(x)) \geq d(\omega_x,\partial \Omega\setminus
%B\left(\omega_x,\kappa d(x,\partial \Omega)\right)).
%\end{displaymath}
%As the right-hand side of the above inequality is bounded from
%below by $\kappa d(x,\partial \Omega)=\kappa d(x,\omega_x)$,
Since $x$ was chosen arbitrarily from $E(\lambda)$, we have thus shown that $E(\lambda)$ is covered by the countable family of balls $B(\omega_i,C r_i)$, $i=1,2,\ldots$, where $C=C(\alpha)= \frac{9}{\alpha}+1$. Using the assumption that $\mu$ is a Carleson measure together with the fact that $\mathcal{H}^2|_{\partial \Omega}$ is $2$-regular, and since the multiplicity of the Whitney balls is controlled by \eqref{eq:Whitney2}, we deduce that
\begin{align*}
\mu(E(\lambda)) 
%\leq \mu\left(\bigcup_i B(\omega_i,Cr_i) \cap \Omega\right) 
\leq \sum_i \mu(B(\omega_i, C r_i)\cap \Omega) \leq \gamma_{\mu} \,\left(\tfrac{9}{\alpha}+1\right)^{2}\, \sum_i r_i^{2} 
%&\leq \gamma_{\mu}\, \left(\tfrac{9}{\kappa}+1\right)^{s\alpha}\,\left( \sum_i r_i^s\right)^{\alpha}\\
\lesssim \sum_i \mathcal{H}^2(B(\omega_i,r_i)\cap \partial \Omega)\overset{\eqref{eq:Whitney2}}{\lesssim} \mathcal{H}^2(U(\lambda)),
\end{align*}
as desired. This concludes the proof of the sufficiency part of Theorem~\ref{thm-main1}.
 
Next, let us prove the necessity part of the assertion. First, we need the solvability of the subelliptic harmonic Dirichlet problem for continuous boundary data. Such a result holds for bounded open sets in $\Hei$ satisfying the uniform outer ball condition, see Remark 3.4 in~\cite{lu} and references therein. Therefore, since every gauge ball $B$ satisfies the uniform outer ball condition and, by assumptions, so is $\Om$, it holds that also $\Om\cap B(x_0, 3r)$ satisfies the condition, see Remark 3.5 in~\cite{lu} for convex open sets. However, the set $\Om\cap B(x_0,3r)$ need not have to be convex, but a ball is convex and, hence, satisfies the uniform outer ball condition. Therefore, the intersection of two sets satisfying uniform outer ball condition also satisfies it, as it suffices to take a smaller radius of those defining outer balls for $\Om$ and $B(x_0,3r)$.

By the discussion at (4.1) in~\cite{lu} we define the Poisson kernel $P=P(x,\om)$ related to $\Om$, for $x\in \Om$ and $\om\in \partial \Om\setminus \Sigma_{\Om}$. Recall, that $\sigma(\Sigma_{\Om})=0$ by (3) in our presentation before the proof of Theorem~\ref{thm-main1}. %This in turn leads us to the following discussion.

Let $\phi:\partial\Om\to\R$ be a continuous function such that $\phi\equiv 4\lambda$ on the set $\partial\Om\cap B(x_0,6r)$, $\phi\equiv 0$ outside the set $\partial\Om\cap B(x_0,7r)$ and $0\le\phi\le 4\lambda$. Then a function $u(y)=\int_{\partial D} P(y,\om)\phi(\om) \ud \sigma(\om)$ is the unique harmonic solution to the Dirichlet Problem in $\Om$ for the Poisson kernel $P$ of domain $\Om$ with boundary data given by $\phi$.  Moreover, we adapt the weak maximum principle in Theorem 8.2.19 (ii) in \cite{blu} and obtain the weak minimum principle for $u$, so that $0\le u\le 4\lambda$ in $\Om$. %(by the standard approach of considering the function $\lambda-u$ harmonic in $\Om\cap\ B(x_0,r)$).
Let us consider a function $w:=4\lambda-u$. Such a function is harmonic in $\Om$, satisfies $0\le w \le 4\lambda$ and has zero boundary values on $\partial\Om\cap B(x_0,6r)$. Therefore, we can use Theorem 1.1 from \cite{gp} to obtain 
\begin{equation*}
\frac{w(x)}{w(A_r(x_0))}=\frac{4\lambda-u(x)}{4\lambda-u(A_r(x_0))}\le c\frac{d(x,\partial\Om)}{r}.
\end{equation*}
Then it follows that
\begin{equation*}
u(x)\ge 4\lambda \left(1-c\frac{d(x,\partial\Om)}{r}\right)+cu(A_r(x_0))\frac{d(x,\partial\Om)}{r},
\end{equation*}
for $x\in\Om\cap B(x_0,r)$ and $c=c(n,\Om)$.

If $x\in\Om\cap B(x_0,\widetilde{r})$ with $\widetilde{r}<\frac{3}{4c}r$, then $1-c\frac{d(x,\partial\Om)}{r}>\frac{1}{4}$. In a consequence, we get
\begin{equation*}
u(x)>\lambda+cu(A_r(x_0))\frac{d(x,\partial\Om)}{r}>\lambda.
\end{equation*}
%provided that $x\in\Om\cap B(x_0,\widetilde{r})$.
In particular, $u>\lambda$  on $\Om\cap B(x_0,\frac{1}{2c}r)$ and so it holds for any ball $B(x_0,r)$ that

\begin{equation}\label{est: thm1-1}
 \mu(B(x_0,r)\cap \Om)\lesssim \mu(B(x_0,\frac{1}{2c}r)\cap\Om)\leq \mu(\{x\in\Om: u(x)> \lambda\})\overset{\eqref{cond:l55}}{\leq} C \sigma(\{\om\in \partial \Om: N_{\alpha}u(\om)>\lambda\}).
\end{equation} 

%Theorem 13 in [CG1998] implies that since $u\geq \lambda>0$ in $D$, then there exists a unique Borel measure on $\partial D$ giving the kernel $K$-representation.

Since $\phi\in C(\partial \Om)$, we have that $\phi\in L^p(\partial\Om, \ud \sigma)$ for any $1\leq p\leq \infty$. This holds, as $\sigma(\partial\Om)<\infty$, due to the $3$-regularity of $\partial \Om$ and its boundedness. Hence, Theorem 1.1 in~\cite{cgn} implies that $\phi\in L^p(\partial\Om, \ud \om^{y})$ for any given $y\in D$. Moreover, it holds that $\ud \om^{y}=P(y,\cdot)\ud \sigma$, for a Poisson kernel. Therefore, Theorem 5.5 (i) in \cite{cgn} implies that 
\begin{equation}\label{est: thm1-2}
 \{\om \in \partial\Om: N_{\alpha}u(\om)>\lambda \}\subset \left\{\om \in \partial\Om: M_{\om^y}(\phi)(\om)>\frac{\lambda}{C}\right\},
\end{equation}
where $M_{\om^y}(\phi)$ stands for the Hardy--Littlewood maximal operator of function $\phi\in L^1 (\partial\Om, \ud \om^y)$ defined in the standard way as follows
\[
 M_{\om^y}(\phi)(\om):=\sup_{0<r<\diam\Om} \frac{1}{\om^{y}(\partial\Om\cap B(\om,r))} \int_{\partial\Om\cap B(\om,r)} |\phi(z)| \ud \om^y(z),\quad \om\in\partial\Om.
\]
Next, we appeal to the following relation between the maximal operator considered with respect to the harmonic measure $\om^y$ and the surface measure $\sigma$, see (6.7) in \cite{cgn}:
\begin{equation}\label{est: thm1-3}
  M_{\om^y}(\phi)(\om)\leq C \left(M_{\sigma}|\phi|^{\beta}\right)^{\frac{1}{\beta}}(\om),\quad \om\in \partial\Om,\quad \hbox{any fixed } y\in\Om.
\end{equation}
The estimate holds for any $1<p\leq \infty$ and $1<\beta<p$. Since $\phi \in L^{\infty}(\partial\Om,\ud \om^y)$ we may choose $p=\infty$.

By \cite{cg98}, pg. 14 it holds that $(\partial\Om, \ud \om^{x}, d_{\partial\Om})$ is the homogeneous space. Thus, by collecting estimates in~\eqref{est: thm1-1}-\eqref{est: thm1-3} and by applying the weak-$L^1$ estimate for doubling spaces in Theorem 3.5.6 in [HKST] and by the definition of $\phi$ we obtain the following estimate 
\begin{align}
 \mu(B(x_0,r)\cap\Om)&\leq C\sigma \left(\left\{\om \in \partial\Om: M_{\om^y}(\phi)(\om)>\frac{\lambda}{C}\right\}\right)  \nonumber \\
& \leq C\sigma \left(\left\{\om \in \partial\Om: M_{\sigma}|\phi|^{\beta}(\om)>\left(\frac{\lambda}{C}\right)^{\beta}\right\}\right) \nonumber \\
& \leq \frac{C^\beta}{\lambda^\beta} \|\phi\|^{\beta}_{L^{\beta}(\partial\Om, \ud \sigma)}\lesssim C \sigma (\partial(B(x_0,3r)\cap \Om)) \lesssim r^3. \label{est-Carl-thm1}
\end{align}

Since $y\in\Om$ is arbitrary and any two harmonic measures $\om^{y}$ and $\om^{y'}$ are comparable for any $y,y'\in\Om$ with the constant depending on the diameter $\diam \Om<\infty$. Thus, $\mu$ is Carleson in $\Om$, as the constants in~\eqref{est-Carl-thm1} do not depend on the choice of $r$.
%\komT{Czy mozna skonstruowac $\phi$ od razu na calym $\Om$ i przyjac $D=\Om$?}
%\komT{Czy $\sigma:=H^2\lfloor \partial \Om$ dla $\Om$- ADP jest $3$-regularna? Tak jest dla kuli, patrz Lemma 2.11 w\cite{af}. Tak jest rowniez dla zbiorow $C^1$ i zadanych gladka funkcja, thm 3.1 w [CGN-AJM] oraz [CGN-MemAMS].}\komM{Tu juz doszlismy, ze powinna byc $3$-reularnosc.}
%\komT{Czy nie potrzeba zauwazyc ze $u$ da sie rozszerzyc do harm w $\Om$ z zachowanie danych brzegowych na $\partial D\cap \partial \Om$?}\komM{A dlaczego?}
\end{proof}

%%\begin{theorem}\label{thm-main2}
%% Let $\Heip\subset \Hei$ be the upper-half space, i.e. $\Heip:=\{(x,y,t)\in \Hei:t>0\}$ and $\mu$ be a Carleson measure in $\Heip$ with the Carleson constant $\gamma_{\mu}$. Then, the assertion of Theorem~\ref{thm-main1} holds for $\mu$ with respect to harmonic functions in $\Heip$.
%%\end{theorem}
%%
%%\komT{1. Mozna zbadac inne polprzestrzenie w $\Hei$, a nawet pas (strip) - patrz Garofalo JMMA. 
%%
%%2. Dowod powinien przejsc bez koniecznosci odwolania sie do j. Poissona zbioru nieograniczonego, przez j. Poissona rodziny polkul zbiegajacych do polprzestrzeni i pokazaniu jedn. zbieznosci reprezentacji poprzez jadra Poissona do subharmonicznej funkcji.}

\section{Carleson measures and M\"obius-type transformations on the unit gauge ball}

The purpose of this section is to show Theorem~\ref{thm-main3}, a counterpart of Lemma 3.3 in Chapter VI.3 in~\cite{ga} characterizing the Carleson measures on the unit disk $\mathbb{D}$ in terms of the canonical M\"obius transformations on $\mathbb{D}$. Namely, the lemma says that \emph{a positive measure $\mu$ on $\mathbb{D}$ is a Carleson measure if and only if the following holds:
\begin{equation}\label{cond:plane}
 \sup_{z_0\in \mathbb{D}} \int_{\mathbb{D}} \frac{1-|z_0|^2}{|1-\bar{z_0}z|^2} \ud \mu(z)=M<\infty.
\end{equation}
Moreover, the constant $M$ is comparable to the Carleson constant, i.e. $M\approx \gamma_{\mu}$ with absolute constants.
}
Notice that, for a given $z_0\in \mathbb{D}$, the integrand in~\eqref{cond:plane} satisfies the following
\begin{equation}\label{eq: mob-plane}
\frac{1-|z_0|^2}{|1-\bar{z_0}z|^2}=\frac{1-\left|\frac{z-z_0}{1-\bar{z_0}z}\right|^2}{1-|z|^2}=\frac{1-|T_{z_0}(z)|^2}{1-|z|^2},
\end{equation}
where $T_{z_0}(z)=e^{-i\theta_0}\frac{z-z_0}{1-\bar{z_0}z}$, for $z_0=r_0 e^{i\theta_0}$ is the M\"obius self-mapping of $\mathbb{D}$ with the property $T_{z_0}(z_0)=0$. Such family of conformal mappings, and its $n$-dimensional counterpart, play an important role in the studies of quasiconformal and quasiregular mappings and related Hardy spaces, see e.g. \cite{AK, AG1, AG2} and~\cite{ahl} for basic properties of such mappings. The relation between expressions in \eqref{eq: mob-plane} and the Carleson condition becomes more apparent once we observe that for small enough radii $r>0$, any $\om\in \partial \mathbb{D}$ and $z\in \mathbb{D} \cap B(\om,r)$ it holds that 
\begin{equation}\label{eq:mob-plane2}
\frac{1-|T_{z_0}(z)|^2}{1-|z|^2} \approx \frac{1-|T_{z_0}(z)|}{1-|z|}\approx \frac{1}{r},
\end{equation}
see Lemma 2.2 in~\cite{AG2}. Hence, \eqref{eq: mob-plane} and \eqref{eq:mob-plane2} together with \eqref{cond:plane} imply the Carleson condition for $\mu$:
$$
\mu(\mathbb{D} \cap B(\om,r)) = r \int_{\mathbb{D} \cap B(\om,r)} \frac{1}{r} \ud \mu\lesssim Mr.
$$
The class of the M\"obius transformations $T_{z_0}$ has no direct counterpart in the Heisenberg setting as a class of the conformal maps from a unit Kor\'anyi--Reimann ball into itself. Nevertheless, recently in~\cite[Section 4.1]{af} the notion of class $T_{z_0}$ has been extended to the subriemannian setting in the following way.

Recall the \emph{Kor\'{a}nyi inversion} in the Kor\'{a}nyi unit sphere centered at the origin defined as follows:
$I(y)=-\frac{1}{\|y\|^4}\left(y_z(|y_z|^2+iy_t), y_t\right)$, where $y=(y_z, y_t) \in \mathbb{H}^1 \setminus \{0\}$. It is the restriction of a conformal self-map of the compactification $\widehat{\mathbb{H}}^1$, with $I(0)=\infty$ and $I(\infty)=0$. Moreover, if $y$ lies in the complex plane (i.e. $y_t=0$), the inversion $I$ agrees with the well-known inversion in the unit circle. 

Let us fix $x\in \mathbb{H}^1$, $a\in \mathbb{H}^1\setminus \{x\}$, and $\rho>0$. Define the map
$T:=T_{x,a,\rho}:\hat{\mathbb{H}}^1 \to \hat{\mathbb{H}}^1$ as follows
\begin{equation}\label{def:Tmob}
T(y):= \delta_{\rho}\left(\left[I(a^{-1}\cdot x)\right]^{-1} \cdot \left[I(a^{-1}\cdot y)\right] \right),
\end{equation}
where $\delta_{\rho}$ denotes the Heisenberg dilation by $\rho$.

Below we collect some properties of maps $T_{x,a,\rho}$ proven in Proposition 4.2, Corollaries 4.9 and 4.11 and Proposition 4.13 in~\cite{af}:
\begin{itemize}
\item[(1)] The mapping
$$
T|_{\mathbb{H}^1 \setminus \{a\}}: \mathbb{H}^1 \setminus \{a\} \to \mathbb{H}^1 \setminus \{\delta_{\rho} \left([I(a^{-1}\cdot x)]^{-1}\right)\}
$$ 
is $1$-quasiconformal. Moreover, $T(x)=0$, $T(a)=\infty$, $T(\infty) =\delta_{\rho}\left([I(a^{-1}\cdot x)]^{-1}\right)$.
\item[(2)] For all $y,y'\in \mathbb{H}^1 \setminus \{a\}$, it holds that 
\begin{align*}
d(T(y),T(y')) = \rho \frac{d(y,y')}{d(a,y) d(a,y')},\quad \|T(y)\| = \rho \frac{d(x,y)}{d(a,y) d(a,x)},\quad J_T(y) = \frac{\rho^4}{d(a,y)^8},
\end{align*}
where $J_T$ denotes the Jacobian of map $T$.
\item[(3)] Let any $x\in B$, $a\in \mathbb{H}^1 \setminus \overline{B}$ and $\rho>0$, be such that
\begin{equation}\label{eq:ImplicitConstants}
\rho \lesssim \min\{d(x,\partial B), d(a,\partial B) \} \quad\text{and} \quad \rho \approx d(a,x).
\end{equation}
Then the map $T=T_{x,a,\rho}$ satisfies
\begin{equation}\label{incl:balls}
B(0,m) \subset T(B) \subset B(0,M).
\end{equation}
for radii $m$ and $M$ depending on $x$, $a$, and $\rho$ only through the implicit multiplicative constants
in the inequalities in \eqref{eq:ImplicitConstants}. 
This property is a reflection of the similar one for the M\"obius self-maps of a unit ball in $\Rn$, see e.g. Lemma 2.2 in~\cite{AG2}

\item[(4)] Let $\omega \in \partial B$, $x\in B$, and $\rho>0$. Assume that $a\in \mathbb{H}^1 \setminus \overline{B}$
and $r>0$  are such that $d(a,\omega) \lesssim r$ and $d(a,B)\geq C r$, for a constant $C>1$. Then, the map $T= T_{x,a,\rho}$ satisfies
\begin{equation}\label{eq:ComparabilityJ}
\frac{d(T(y),\partial T (B))}{d(y,\partial B)} \approx_C \frac{\rho}{d(y,a)^2},\quad \text{for all }y\in B(\omega,r)\cap B.
\end{equation}
The similar property holds in $\Rn$, see Lemma 2.2 in \cite{AG2}.
\end{itemize}

After the above preparatory observations we are in a position to prove the main result of this section. For the readers convenience we recall the statement of Theorem~\ref{thm-main3} (cf. Section 1).
\smallskip
%\begin{theorem}[cf. Lemma 3.3 in \cite{ga}]\label{thm-main3}
\\
 {\bf Theorem~\ref{thm-main3}.}  \emph{A measure $\mu$ on the unit gauge ball $B:=B(0,1)\subset \Hei$ is a Carleson measure if and only if
 \begin{equation}\label{cond: thm-m3}
  \int_{B} \left( \frac{d(T_{x,a,\rho}(y),\partial T_{x,a,\rho} (B))}{d(y,\partial B)}\right)^3 \ud \mu(y)=M<\infty,
 \end{equation}
 for all $x\in B$, $a\in \mathbb{H}^1 \setminus \overline{B}$, and $\rho>0$ such that 
$\rho \lesssim \min\{d(x,\partial B), d(a,\partial B) \}$ and $\rho \approx d(a,x)$.
}
%\end{theorem}
\smallskip
\\
Basing on~\eqref{eq:mob-plane2}, one could expect that the corresponding hypotheses~\eqref{cond: thm-m3} of the above theorem in $\Hei$ should involve the Kor\'anyi norms of points in $y\in B$ and their images $T(y)$. Indeed, by property~\eqref{incl:balls} we have that
 \[
\left(\frac{m+1}{2}\right) \frac{1-\|T_{x,a,\rho}(y)\|}{1-\|y\|} \leq \frac{1-\|T_{x,a,\rho}(y)\|^2}{1-\|y\|^2}\leq (M+1) \frac{1-\|T_{x,a,\rho}(y)\|}{1-\|y\|}.
 \]
However, due to the geometry of balls in $\Hei$ it is more convenient to work with the distances to the corresponding boundaries of $y$ and $T(y)$, see the proof below.

\begin{proof}[Proof of Theorem~\ref{thm-main3}]
Set $T:=T_{x,a,\rho}$ and observe that by~\eqref{eq:ComparabilityJ} and by assumptions the following holds for any $y\in B(\omega,r)\cap B$:
 \begin{equation}\label{est2-thm-m3}
  \frac{d(T(y),\partial T (B))}{d(y,\partial B)} \approx_C \frac{\rho}{d(y,a)^2}\approx \frac{d(a,x)}{d(y,a)^2}.
 \end{equation}
 Choose $x$ such that $x\in B(\om,r)\cap B$ and $d(\om, x)=\frac{r}{2}$. Then, for $a$ choosen as in~\eqref{eq:ComparabilityJ}, it holds that
 \begin{align*}
  r&\leq d(a,y) \leq d(a,x)+d(x,y) \approx d(x,\partial B)+\frac32 r\approx \frac52 r,\\
  &\,\,\,\,\,\,\,d(a,x)\approx d(x,\partial B)\approx d(x,\om)\approx r.
 \end{align*}
Hence, by applying these estimates in~\eqref{est2-thm-m3} we obtain that
 \begin{equation}\label{est-thm-m3}
  \frac{d(T(y),\partial T (B))}{d(y,\partial B)} \approx \frac{1}{r} \approx \frac{1}{d(x,\partial B)}.
 \end{equation}
% \komT{W powyzszej nierownosci wystarczy $\lesssim$.}
 
 We are ready to show the necessity part of the assertion. Let us assume that \eqref{cond: thm-m3} holds. Then, for any $\om\in \partial B$ and $r>0$ we have
 \begin{align*}
  \mu(B(\om, r)\cap B) &=\int_{B(\om, r)\cap B}\frac{d(x,\partial B)^3}{d(x,\partial B)^3} \ud \mu(y)  \\
  &\approx d(x,\partial B)^3 \int_{B(\om, r)\cap B}\left(\frac{d(T(y),\partial T (B))}{d(y,\partial B)}\right)^3\ud \mu(y) \\
  &\lesssim M d(x,\partial B)^3 \approx r^3.
 \end{align*}
 In order to show the opposite implication in the assertion of the theorem let us consider two cases for points $x\in B$ in the definition of maps $T=T_{x,a,\rho}$ : (1) $d(x,\partial B)>\frac14$, and (2) $d(x,\partial B)\leq \frac14$. In the first case by~\eqref{est-thm-m3}, we trivially have that
 \[
  \int_{B(\om, r)\cap B}\left(\frac{d(T(y),\partial T (B))}{d(y,\partial B)}\right)^3\ud \mu(y)\lesssim_C \mu(B) \leq C\gamma_{\mu}<\infty.
 \]
Therefore, we may assume that points $x\in B$ satisfy $d(x,\partial B)\leq \frac14$ in which case we mimic the approach in the proof of Lemma 3.3 in~\cite[Section 3, Chapter VI]{ga}. However, we need to take into account the differences between the Euclidean and the Heisenberg settings. 

Recall that the Euclidean radial curves need not be horizontal in $\Hei$ and hence may have an infinite subriemannian length.
However, by works \cite{kr} and \cite{bt}, see also the discussion in Section 2.1.2 in~\cite{af}, we have that the following formula describes the radial curves given by the horizontal curves joining the origin with the point $\om=(z,t)$ belonging to the boundary $\partial B\setminus \{z=0\}$: 
\begin{equation}\label{def:rad-curv}
\gamma(s,(z,t)) = \left(sz e^{-\mathrm{i}\frac{t}{|z|^2}\log s},s^2 t\right),\quad (z,t)\in \partial B\setminus \{z=0\}.
\end{equation}
It is easy to compute that $\|\gamma(s)\|=s$. Moreover, given $x\in B$ we may find a point $\om\in \partial B$ corresponding to $x=(x_z, x_t)$, in a sense that $x=\gamma(s, \om)$ for some $0<s<1$, by solving \eqref{def:rad-curv} for $z$ and $t$. Namely, we have that 
\[
 t=\frac{x_t}{\|x\|^2},\quad z= \frac{x_z}{\|x\|}\,e^{\mathrm{i}\frac{x_t}{|x_z|^2}\,\log \|x\|}. 
\]
We denote such point by $\om_x$ and define the following family of subsets in $B$:
\[
 E_n:=\{y\in B: d(y, \om_x)<2^n d(x,\partial B)\}\quad n=1,2,\ldots.
\]
Therefore, since $\mu$ is assumed to be a Carleson measure in $B$, we find that
\[
 \mu(E_n)\leq\mu\big(B(\om_x, 2^n d(x,\partial B)) \cap B\big)\leq \gamma_{\mu}2^{3n} d(x,\partial B)^3,\quad  n=1,2,\ldots.
\]
Hence, by appealing to~\eqref{est2-thm-m3}, we obtain the following estimate
\begin{align}
 &\int_{B(\om, r)\cap B}\left(\frac{d(T(y),\partial T (B))}{d(y,\partial B)}\right)^3\ud \mu(y) \nonumber \\
 &\,\,\,\,\leq \int_{E_1}\left(\frac{d(T(y),\partial T (B))}{d(y,\partial B)}\right)^3\ud \mu(y)+\sum_{n=2}^{\infty}\int_{E_n\setminus E_{n-1}}\left(\frac{d(T(y),\partial T (B))}{d(y,\partial B)}\right)^3\ud \mu(y). \label{est3-thm-m3}
% &\,\,\,\,\leq \frac{\mu(E_1)}{2^{3} d(x,\partial B)^3}+\sum_{n=1}^{\infty}\mu(E_n)\left(\frac{d(a,x)}{d(y,a)^2}\right)^3  \leq \sum_{n=1}^{\infty}\mu(E_n)\left(\frac{d(a,x)}{d(y,a)^2}\right)^3.
\end{align}
Since, by assumption $d(a,x)\approx d(x,\partial B)$ and for $y\in E_n\setminus E_{n-1}$ it holds that 
$$
2^{n-1}d(x,\partial B)<d(y,\om_x)<2^nd(x,\partial B),
$$
we have that
\[
\frac{d(a,x)}{d(y,a)^2}\lesssim \frac{d(x,\partial B)}{d(y,\om_x)^2}  \lesssim \frac{1}{2^{2n} d(x,\partial B)},
\]
as $d(y,a)>d(y, \om_x)$ due to the assumption that $a\in \Hei\setminus \overline{B}$. We are in a position to complete the above estimate~\eqref{est3-thm-m3} as follows (cf.~\eqref{est2-thm-m3}):
\begin{align*}
 &\int_{E_1}\left(\frac{d(T(y),\partial T (B))}{d(y,\partial B)}\right)^3\ud \mu(y)+\sum_{n=2}^{\infty}\int_{E_n\setminus E_{n-1}}\left(\frac{d(T(y),\partial T (B))}{d(y,\partial B)}\right)^3\ud \mu(y)\\
 & \leq \frac{\mu(E_1)}{2^{6} d(x,\partial B)^3} + \sum_{n=2}^{\infty}\int_{E_n\setminus E_{n-1}}\frac{1}{2^{6n}d(x,\partial B)^3} \ud \mu(y)\\
 & \lesssim \sum_{n=1}^{\infty} \frac{1}{2^{6n}d(x,\partial B)^3} \mu(E_n) \lesssim_{\gamma_{\mu}} \sum_{n=1}^{\infty} \frac{1}{2^{3n}}<\infty.
\end{align*}
This completes the sufficiency part of the proof and thus, the whole proof is completed as well.
\end{proof}

\begin{remark}
 We observe that since the Kor\'{a}nyi inversion can be defined in groups $\Hn$, see e.g.~\cite{kr2}, so is the class of maps $T_{x,a,\rho}$. Moreover, the horizontal curves~\eqref{def:rad-curv} exist not only in $\Hei$, but also in $\Hn$ (in fact, in polarizable groups, see Section 3 in~\cite{bt}). Therefore, Theorem~\ref{thm-main3} has a counterpart in $\Hn$. However, for the sake of simplicity of the presentation and in order to avoid repeating similar construction of maps $T_{x,a,\rho}$ presented in~\cite[Section 4.1]{af}, we decided to state the theorem only in the $\Hei$ setting. 
\end{remark}

\section{Square function and Carleson measures for $L^2$ and BMO boundary data}

The purpose of this section is to prove main results of this work, namely Theorems~\ref{thm-jk91} and~\ref{thm-jk96}, which generalize, respectively, Theorem 9.1 and Theorem 9.6 in~\cite{jk}. Theorem 9.1 provides the upper bound for the $L^2$-norm of the square function in terms of the $L^2$-norm of the boundary data on NTA domains. Theorem 9.6 gives a Carleson-measure estimate for a subelliptic harmonic function defined by the integral of a BMO function with respect to harmonic measure on the NTA domain. 
%\komM{Mozna napisac, ze jest to potrzebne do $\varepsilon$-aproksymowalnosci, ale moze lepiej nie zdradzac, ze chcemy to wykorzystac do kolejnego dowodu.}
Such estimates in $\Rn$ are essential, for example, when proving the so-called $\varepsilon$-approximability for harmonic functions, i.e. the existence of a BV function $v$ such that for a harmonic function $u$ we have $\|u-v\|_{\infty}<\varepsilon$ and $|\nabla v(y)|\ud y$ is a Carleson measure, see e.g.~\cite{bh, gmt, ht} . 

Recall Definition~\ref{def: square} of the square function (also known in the literature as the area function, depending on the authors and the context): 
\[
S_{\alpha}u(x)^2:=\int_{\Gamma_{\alpha}(x)}|\nabla_H u(y)|^2 d(y,\partial \Om)^{2-Q}\ud y.
\]
For the readers convenience we recall our main results (cf. Section 1). 
\smallskip
%\begin{theorem}[$L^2$-boundedness of the square function]\label{thm-jk91}
\\
 {\bf Theorem~\ref{thm-jk91}.} \emph {Let $\Om\subset \Hn$ be a bounded NTA domain. Let further $f\in L^2(\ud \om)$ and $u(x):=\int_{\partial \Om} f(y)\ud \om^x(y)$. Then, the following estimate holds for the square function $S_{\alpha}$ of a harmonic function $u$ in $\Om$
 \[
  \|S_{\alpha}u\|_{L^2(\ud \om)}\leq C\|f\|_{L^2(\ud \om)}.
 \]
 }
%\end{theorem}
The corresponding Euclidean result in~\cite{jk}, cf. Theorem 9.1, is proven for any $1<p<\infty$. However, for us the case $p=2$ is the most interesting, as it is the one that we would like to use to prove $\varepsilon$-approximability for harmonic functions. According to our best knowledge, the result is new in the subriemannian setting. 
%\komT{zmienic, bo tekst jest we wstepie:
%Our next result is the subriemannian counterpart of the Euclidean result, i.e. Theorem  9.6 in~\cite{jk}. Moreover, it also generalizes Theorem 3.4 in~\cite{ga} for the unit disc in the plane, see Remark~\ref{rem-Garnett-34}. We further refer to Example~\ref{ex-ball} for the case of the unit gauge ball in $\Hn$, where the Green function in the assertion of Theorem~\ref{thm-jk96} can be explicitly estimated form below in terms of the distance function, thus providing more classical and handy estimate.
%}
\smallskip
\\
{\bf Theorem~\ref{thm-jk96}.} \emph{ Let $\Om\subset \Hn$ be a bounded NTA domain and $u$ be a subelliptic harmonic in $\Om$ such that $u(x)=\int_{\partial \Om} f(y)\ud \om^x(y)$ for some $f\in BMO(\partial \Om)$. Then, for any ball $B(x_0, r)$ centered at $x_0\in \partial \Om\setminus \Sigma_{\Om}$ with any $0<r<r_0\leq \min\{1, \frac{d(x_0,\Sigma_{\Om})}{M}\}$ it holds that
 \[
 \int_{B(x_0,r)\cap \Om} |\nabla_H u|^2 G(x, A_r(x_0))\ud x\leq C\om(B(x_0,r)\cap \partial \Om),
 \]
 where constant $C$ depends on $n, M, r_0$ and $\|f\|_{BMO(\partial \Om)}$. 
 }
 %\komT{Czy mozna sie uwolnic od $A_r(x_0)$ i zastapic dowolnym ustalonym punktem $y\in \Om$ poza kazda kula $B(x_0, r)$?} 
% \end{theorem}

Let us remark that the assertion of the theorem can be formulated equivalently in a way similar to the bottom of page 3 in~\cite{ht}, i.e. by using the supremum over radii and the averaged integral.

In the theorem, $\om$ stands for a harmonic measure with respect to any but fixed point $y\in \Om$ and so $\om:=\om^y$. However, for the sake of convenience of the presentation in what follows we omit the reference points. This is justified by Observation~\ref{obs-bmo} and by following the standard notation convention for harmonic measures, see e.g.~\cite{jk}. 

%Since the harmonic measure on the boundary of an NTA domain is $Q\!-\!1$-Ahlfors regular for $Q=2n+2$, the following consequence the above theorem holds. \komT{czy dlatego czy trzeba zalozyc ADP? sprawdzic}\komM{ To nie jest prawda bez zalozenia, ze dziedzina ma regularny brzeg. Istnieja dziedziny, ktore sa NTA, a nie sa chord arc. Ale samo to nie pomaga, bo a priori nie wiemy, czy mozemy porownac miare harmoniczna z Tw. 5.2 z miara powierzchniowa, ktora wystepuje w definicji regularnosci. Dlatego powinnismy zalozyc, ze dziedzina jest ADP i gladka, by miec porownywalnosc obu miar.}

Upon strengthening the regularity assumptions of the boundary, the following consequence of Theorem~\ref{thm-jk96} holds, as the ADP condition allows us to compare the harmonic measure with the surface measure, see (1) in the discussion following Definition~\ref{defn:char}.

\begin{cor} Under the assumptions of Theorem~\ref{thm-jk96} if we additionally assume that $\Om$ is a smooth ADP domain, then it holds for any ball $B(x_0, r)$ centered at $x_0\in \partial \Om\setminus \Sigma_{\Om}$ with radius $0<r<r_0\leq \min\{1, \frac{d(x_0,\Sigma_{\Om})}{M}\}$ that
 \[
 \int_{B(x_0,r)\cap \Om} |\nabla_H u|^2 G(x, A_r(x_0))\ud x\leq C r^{Q-1},
 \]
where $C$ depends on $n, M, r_0$ and $\|f\|_{BMO(\partial \Om)}$. 
\end{cor}

An important class of examples of NTA domains in $\Hn$ is the one of (Euclidean) $C^{1,1}$ domains, see~\cite{cg98, gp}. Since the Green function $G(x,\cdot)$ is nonnegative subelliptic harmonic in $\Om\setminus \{A_r(x_0)\}$ we may use Theorem 1.2 in~\cite{gp} to get the following lower boundary Harnack-type estimate for a (Euclidean) $C^{1,1}$ domain $\Om\subset \Hn$, any $x_0\in \partial \Om\setminus \Sigma_{\Om}$ and $0<r'<\min\{1, d(x_0,\Sigma_{\Om})/M, d(A_r(x_0), \partial \Om)\}$ and for all $x\in B(x_0,r')\cap \Om$:
\[
 G(x,A_r(x_0))\geq  C(n,\Om) G(A_{r'}(x_0),A_r(x_0)) \frac{d(x,\partial \Om)}{r'}\gtrsim_{C(n,\Om)}  G(A_{r'}(x_0),A_r(x_0)) d(x,\partial \Om).
\]
Moreover, notice that by building a chain of balls joining $A_{r'}(x_0)$ with fixed, but any $y\in \Om$ such that $\dist(y,\partial \Om)>r_0$, for $r_0$ as in the definition of NTA domains, and by the standard iteration of the Harnack inequality on metric balls (see e.g.~\cite[Corollary 5.7.3]{blu}), we obtain that
\[
 G(A_{r'}(x_0),A_r(x_0)) \geq C^{N} G(y, A_r(x_0)),
\] 
where the length of the Harnack chain $N$ depends on $\diam \Om$ and $r$ and the Harnack constant C depends on the geometric parameters of $\Hn$ and $\Delta_H$, cf.~\cite{blu}.

The above discussion and Proposition 5.6 in~\cite{gp} imply the following Carleson-type estimate.
\begin{cor}\label{cor-52}  Let $\Om\subset \Hn$ be a (Euclidean) $C^{1,1}$ domain which also satisfies the uniform outer ball condition. Then, under the assumptions of Theorem~\ref{thm-jk96}, it holds for any ball $B(x_0, r)$ centered at $x_0\in \partial \Om\setminus \Sigma_{\Om}$ with radius $0<r<r_0\leq \min\{1, \frac{d(x_0,\Sigma_{\Om})}{M}\}$ that
 \begin{equation}\label{est-Carleson-G}
 \int_{B(x_0,r)\cap \Om} |\nabla_H u|^2 G(A_{r'}(x_0),A_r(x_0)) d(x,\partial \Om) \ud x\leq C \om(B(x_0,r)\cap \partial \Om),
 \end{equation}
 where $C$ depends on $n, M, r_0, \diam \Om$ and $\|f\|_{BMO(\partial \Om)}$m and the radius $r'$ satisfies 
 $$
 0<r'<\min\{1, d(x_0,\Sigma_{\Om})/M, d(A_r(x_0), \partial \Om)\}.
 $$
 \end{cor}
 
% \komT{Nie jestem pewien wartosci takiego sformulowania. Lepiej jest kiedy funkcja Greena pozostaje pod calka. G zanika do zera na $\partial \Om$, wiec takie sformulowanie mowi, ze gradient przemnazamy przez wyrazenie, ktore na brzegu jest zero:
% \[
% \int_{B(x_0,r)\cap \Om} |\nabla_H u|^2 G(A_{r'}(x_0),A_r(x_0)) d(x,\partial \Om) \ud x\leq C r^{Q-1}.
% \]
% }

Let us observe some further consequences of Theorem~\ref{thm-jk96}. First, we explain how it corresponds to Garnett's result, cf. ~\cite[Theorem 3.4]{ga}. Then, in Example~\ref{ex-ball} we show how for the gauge unit ball, a special but important case of the NTA domain in $\Hn$, the estimate in the theorem takes simpler and convenient form.

\begin{remark}\label{rem-Garnett-34}
Theorem~\ref{thm-jk96} generalizes part of the following characterization of the Carleson measures on the unit disc in the plane to the setting of NTA domains in $\Hn$, cf. Theorem 3.4 in~\cite{ga}:
\vspace{0.1cm}

\emph{Let $\phi\in L^1(\mathbb{S}^1)$ and $u$ be a Poisson extension of $\phi$ to the unit disc in $\R^2$. Let
\[
 \ud \lambda_{\phi}:=|\nabla \phi(z)|^2 \ln \frac{1}{|z|}\,\ud x\ud y.
\]
Then $\phi \in BMO(\mathbb{S}^1)$ if and only if $\lambda_{\phi}$ is a Carleson measure. Moreover, the Carleson constant of $\lambda_{\phi}$ is comparable to $\|\phi\|^2_{BMO}$.
}
\vspace{0.1cm}

Recall that, up to the constant $\frac{1}{2\pi}$, the function $\ln \frac{1}{|z|}$ is the Green function of the planar unit disc with a pole at $0$; also that Euclidean balls are NTA domains. Moreover, the harmonic measure $\om \approx \sigma$, where $\sigma$ stands for the surface measure on $\mathbb{S}^1$, see Ex. 3, Ch. I in~\cite{ga}. Therefore, the sufficiency part of the above theorem corresponds in $\R^2$ to the assertion of Theorem~\ref{thm-jk96}.

\end{remark}

The next consequence of Theorem~\ref{thm-jk96} addresses the fact that for some NTA domains in $\Hn$ the Green functions can be found explicitly and so the Carleson condition in Theorem~\ref{thm-jk91} can be refined.

\begin{ex}\label{ex-ball}
Let $\Om=B(0, 1)$ be a unit gauge ball in $\Hn$. By Corollary 4 in~\cite{cg98} such balls are NTA domains. Below, we show that on $\Om$ it is possible to refine the estimate~\eqref{est-Carleson-G} for $G(A_{r'}(x_0),A_r(x_0))$ and obtain the following more natural Carleson estimate.

Recall that for the unit gauge ball in $\Hn$ the set of characteristic points $\Sigma_{\Om}$ consists of the north and south poles.
\smallskip
\\
 \emph{Fix $\delta \in (0,1)$ and let $u$ be as in Theorem~\ref{thm-jk96}, i.e. a harmonic function on $\Om$ with the boundary data in BMO. Then, for all points $x_0=(z,t)\in \partial B\setminus \{z: |z|\leq \delta\}$ and all radii $0<r<r_0\leq \min\{1, \frac{d(x_0,\Sigma_{\Om})}{M}\}$ it holds 
 \begin{equation}\label{est-Carl-true}
 \int_{B(x_0,r)\cap \Om} |\nabla_H u|^2 d(x,\partial \Om) \ud x\leq C r^{Q-1}.
 \end{equation}
 Here, the constant $C$ is as in Corollary~\ref{cor-52} and, additionally, depends on $\delta$.
 }
 \smallskip
 \\
\noindent In order to show this estimate, we appeal to the horizontal curves joining the origin with the point $x_0=(z,t)$ in the boundary $\partial B\setminus \{z=0\}$, see~\eqref{def:rad-curv}. The proofs of Lemmas A.2 and A.4 in~\cite{af} show the following properties of curves $\gamma_{x_0}$:
\begin{align*}
& d(\gamma_{x_0}(s), \gamma_{x_0}(s'))\leq {\rm length}(\gamma_{x_0}(\cdot)|_{[s,s']})=\frac{s'-s}{|z|},\quad 0<s<s'\leq 1, \\
& d(\gamma_{x_0}(s), \partial B)\gtrsim \frac{1-s}{|z|},\, \hbox{ if}\quad 1-s\leq |z|. 
\end{align*}

These properties allow us to choose $s$ such that $1-s=r|z|$ and obtain a point on $\gamma_{x_0}$ which satisfies the definition of a corkscrew point $A_r(x_0)$ in the interior corkscrew condition in Definition \ref{def:NTA}. Choose $r'$ close enough to $r$ (i.e. $|r-r'|\ll 1$) and the corresponding 
$s'$ with $1-s'=r'|z|$. Therefore, we get that
\[
 d(A_r(x_0), A_{r'}(x_0))=d(\gamma_{x_0}(s), \gamma_{x_0}(s'))\leq \frac{|s'-s|}{|z|}\leq \frac12 r' \lesssim_{M} \frac12  d(A_{r'}(x_0), \partial B),
\]
where $M$ stands for the NTA constant of a gauge unit ball in $\Hn$. In a consequence, we may estimate function $G$ from below as follows:
\[
G(A_{r'}(x_0),A_r(x_0))\gtrsim \frac{1}{d(A_r(x_0), A_{r'}(x_0))^{Q-2}}\gtrsim \left(\frac{|z|}{|s-s'|}\right)^{Q-2} \gtrsim |z|^{Q-2},
\]
and the proof of the first inequality follows by repeating the steps of the corresponding proof of Property (1.9) in Theorem 1.1 in~\cite{gw}, see Proposition A.1 in Appendix. From this, the estimate~\eqref{est-Carl-true} follows immediately, by applying Corollary~\ref{cor-52} upon noticing that under our assumptions on $x_0$, it holds that $|z|^{2-Q}$ remains bounded from the above by $\delta^{2-Q}$.

%\gamma(s,(z,t)) = \left(sz e^{-\mathrm{i}\frac{t}{|z|^2}\log s},s^2 t\right),\quad (z,t)\in \partial B\setminus \{z=0\}.
%For a given $x_0\in \partial \Om$ and small enough radius $r$, consider the Green function $G(x,A_r(x_0))$ of $\Om$. Then, we may check that
%\[
% G(x,A_r(x_0))=c(Q)\left(\frac{1}{\|A_r(x_0)^{-1}\cdot x\|^{Q-2}}-h^{A_r(x_0)}(x)\right)\geq 
% c(Q)\left(\frac{1}{\|A_r(x_0)^{-1}\cdot x\|^{Q-2}}-\frac{(cM)^{Q-2}}{r^{Q-2}}\right),
%\]
%by the maximum principle applied to the harmonic correction term $h^{A_r(x_0)}$. Since $\|A_r(x_0)^{-1}\cdot x\|\leq d(A_r(x_0), x)$, by restricting further the radius of a ball in the Carleson condition $r' \approx \frac{r}{5cM}$ we get that $G(x, A_r(x_0))\gtrsim d(x,\partial \Om)$ and the measure $|\nabla_H u(x)|^2 d(x, \partial \Om) \ud x$ is the Carleson measure on $\Om$.
\end{ex}

%Similar discussion gives us also the corresponding assertion on half-spaces in $\Hn$. \komT{Czy mam racje?}\komM{Kiedys o tym rozmawialismy. Odpowiedz jest, ze raczej nie, bo nie ma dobrych wspolrzednych zadanych przez rodzine krzywych, tak jak w przypadku kuli.}

\subsection{Proof of Theorem~\ref{thm-jk91}}

Before proving Theorem~\ref{thm-jk91} we need a counterpart of Theorem 5.14 in~\cite{jk} in $\Hn$. Here we present it in a weaker form, i.e. only one implication, cf.~\cite{jk}.

\begin{prop}\label{prop-thm514}
 Let $\Om \subset \Hn$ be an NTA domain. Let $f \in L^2 (\partial \Om, \ud \om^{z})$ for some $z\in \Om$ and such that $\int_{\partial \Om} f\ud \om^z=0$. Then a function $u(x):=\int_{\partial \Om} f(y)\ud \om^x(y)$ satisfies the following identity
 \begin{equation}\label{ident-prop513}
  \int_{\Om} |\nabla_H u|^2 G(x,z)\ud x=\frac12 \int_{\partial \Om} f(y)^2\ud \om^z(y) <\infty.
 \end{equation}
\end{prop}

\begin{proof}
The proof closely follows the corresponding one in \cite{jk} and, therefore, we discuss only the main steps. The key tool used in \cite{jk} is the Riesz representation theorem for subharmonic functions in $\Rn$, whose subriemannian counterpart is given by Theorem 9.4.7 in~\cite{blu}, applied to $-u$ in the notation of \cite{blu}; see also Definitions 9.4.1 and 9.3.1 in \cite{blu}. Namely, the following holds.

\emph{Let $v$ be a subharmonic function in a domain $\Om\subset \Hn$. Then $\int_{\Om} G(x,z)\Delta_H u(x) \ud x<\infty$ for some $z\in \Om$ if and only if  $v$ has a harmonic majorant. Moreover, if $h$ denotes the least harmonic majorant of $v$, then it holds
\begin{equation}\label{prop54-aux1}
 v(x)=h(x)-\int_{\Om} G(x,y)\Delta_H v(y)\ud y.
\end{equation}
}
For the proof of Proposition~\ref{prop-thm514} one defines a subharmonic function $v=u^2$, as $\Delta_H v=2|\nabla_H u|^2\geq 0$ and applies the above representation theorem. Moreover, $v(z)=u^2(z)=0$ by assumptions of the proposition. Since Green's function is zero at the boundary of $\Om$ we have, by~\eqref{prop54-aux1}, that $h\equiv f^2$ on $\partial \Om$, and so $h(z)=\int_{\partial \Om} f(y)^2\ud \om^z(y)$. Upon collecting these observations we obtain~\eqref{ident-prop513}. Then, as in~\cite{jk}, we assume that $f\in L^2(\partial \Om, \ud \om^z)$ with $\int_{\partial \Om} f\ud \om^z=0$ and approximate $f$ in the $L^2(\partial \Om, \ud \om^z)$-norm by the sequence of continuous functions, see the proof of Theorem 5.14 in~\cite{jk} for the remaining details.
\end{proof}

\begin{proof}[Proof of Theorem~\ref{thm-jk91}]
 Since a function in $L^2(\partial \Om, \ud \om)$ can be approximated by $C(\partial \Om)$-functions we may assume that $f\in C(\partial \Om)$. Moreover, without the loss of generality, we may also assume that $f>0$, as otherwise we split $f$ into a positive and negative parts and consider two cases separately. 
 
In what follows we will employ the Green function $G$ of domain $\Om$ and so, in order to avoid problems with the pole of $G$ we need to bring on stage the truncated square function, i.e. the operator $S_{\alpha}u(x)$ considered with respect to truncated cones $\Gamma_{\alpha}^h(x):=\Gamma_{\alpha}(x)\cap B_d(x, h)$, for $x\in \partial \Om$ and $h$ small enough, so that the pole of $G$ at $z\in \Om$ does not belong to any of such truncated cones. Thus, we split $S_{\alpha} u(x)$ as follows 
\begin{equation}\label{S-split}
S_{\alpha}u(x)^2=\int_{\Gamma_{\alpha}^h(x)}\,+\,\int_{\Gamma_{\alpha}(x)\setminus \Gamma_{\alpha}^h(x)}.
\end{equation}
The second integral can be handled by the gradient estimates for harmonic function $u$ (see~\cite[Proposition 2.1]{lu}) and by the Harnack inequality as follows:
\begin{align}
 \int_{\Gamma_{\alpha}(x)\setminus \Gamma_{\alpha}^h(x)}|\nabla_H u(y)|^2d(y,\partial \Om)^{2-Q}\ud y & \leq
  c(n) \int_{\Gamma_{\alpha}(x)\setminus \Gamma_{\alpha}^h(x)} \frac{1}{d(y, \partial \Om)^2} \Big(\sup_{B_d(y,\frac14 d(y, \partial \Om))} |u| \Big)^2\,d(y,\partial \Om)^{2-Q}\,\ud y \nonumber \\
  & \lesssim_{n, M, C} u^2(z)\int_{\Gamma_{\alpha}(x)\setminus \Gamma_{\alpha}^h(x)} \frac{1}{d(y, \partial \Om)^Q} \,\ud y \label{est-hi-thm61}\\
  & \lesssim_{n, M, C} u^2(z) \bigg(\frac{h}{1+\alpha}\bigg)^{-Q} |\Om|. \nonumber
\end{align}

The Harnack inequality (see e.g.~\cite[Corollary 5.7.3]{blu})) is used in the second estimate: we choose big enough compact subset of $\Om$ containing points $z$ and $y$. Since they both are enough far away from the boundary, there exists a Harnack chain of finite length, depending on $M$, joining $z$ and $y$. We iterate the Harnack estimate along that chain and obtain~\eqref{est-hi-thm61} with constant $C$ coming from the constants in the Harnack inequality.

Therefore, since by assumptions $u(z)=\int_{\partial\Om}f(y)\ud \om^{z}(y)$, we get the estimate
\[
\int_{\Gamma_{\alpha}(x)\setminus \Gamma_{\alpha}^h(x)}|\nabla_H u(y)|^2d(y,\partial \Om)^{2-Q}\ud y \lesssim_{n, M, C} \bigg(\frac{h}{1+\alpha}\bigg)^{-Q} |\Om| \|f\|_{L^2(\om^z)}^2.
\]
In order to get the $L^2(\partial \Om)$-norm estimate we integrate both sides of the above inequality and use the fact that harmonic measure is a probability measure to obtain
\begin{align}
\int_{\partial \Om}\int_{\Gamma_{\alpha}(x)\setminus \Gamma_{\alpha}^h(x)}|\nabla_H u(y)|^2d(y,\partial \Om)^{2-Q}\ud y \ud \om^z &\lesssim_{n, M, C} \bigg(\frac{h}{1+\alpha}\bigg)^{-Q} |\Om| \|f\|_{L^2(\om^z)}^2 \nonumber \\
&\lesssim_{n, M, C, \alpha, \diam \Om, h}\|f\|_{L^2(\om^z)}. \label{est-non-trunc}
\end{align}

We now proceed to estimate the first integral in~\eqref{S-split}. For any point $y\in \Om$ let us denote by $q_y$ a point at which the distance $d(y,\partial\Om)$ is attained. Next, observe that a point $y\in \Gamma_{\alpha}(x)$ if and only if $x\in S(y)$ the shadow of point $y$, defined as $S(y):= \partial \Om\cap B(y, (1+\alpha) d(y,\partial \Om))$. For any $z\in S(y)$ it holds that 
\[
d(x,z)\leq d(x,y)+d(y,z)\leq 2(1+\beta)d(y,\partial \Om).
\]
Therefore, $\left\{x\in \partial \Om: y\in \Gamma_{\alpha}^h(x)\right\}\subset \Delta(q_y, 2(1+\alpha)d(y,\partial \Om))$.

We set $\Om_h=\{y\in\Om: d(y,\partial\Om)<h\}$. An application of the Fubini theorem together with the Dahlberg-type estimate in Theorem~\ref{cg-thm21}, give us that
\begin{align}
& \int_{\partial \Om}\int_{\Gamma_{\alpha}^h(x)}|\nabla_H u(y)|^2d(y,\partial \Om)^{2-Q}\ud y \ud \om^z(x) \nonumber \\
&\phantom{AAA} = \int_{\Om_h} |\nabla_H u(y)|^2 d(y, \partial \Om)^{2-Q} \om^z\left(\left\{x\in \partial \Om: y\in \Gamma_{\alpha}^h(x)\right\}\right) \ud y \nonumber \\
&\phantom{AAA} \leq \int_{\Om_h} |\nabla_H u(y)|^2 d(y, \partial \Om)^{2-Q} \om^z \left(\Delta(q_y, 2(1+\alpha)d(y,\partial \Om))\right)\,\ud y \label{thm13-aux1}\\
&\phantom{AAA} \leq \int_{\Om_h} |\nabla_H u(y)|^2 G(y, z)\,\ud y.\label{thm13-aux2}
\end{align}
The last inequality follows by standard reasoning which, however, deserves some details.

Let $z\in\Om\setminus B_d(q_y, 2M(1+\alpha)d(y,\partial\Om))$. Then by Theorem~\ref{cg-thm21}  we have (see also Section 5 in \cite{cgn} to see why constant $a$ in~\cite[Theorem 1]{cg98} can be taken equal to $M$):
\begin{align}
\om^z\big(\Delta(q_y,2(1+\alpha)d(y,\partial\Om))\big) &\approx\frac{|B_d(z,2(1+\alpha)d(y,\partial\Om))|}{(2(1+\alpha)d(y,\partial\Om))^2}G(z,A_{2(1+\alpha)d(y,\partial\Om)}(q_y)) \nonumber \\
&\approx_{\alpha} d(y,\partial\Om)^{Q-2}G(z,A_{2(1+\alpha)d(y,\partial\Om)}(q_y)). \label{thm13-aux3}
\end{align}
By taking $h$ small enough we ensure that $2(1+\alpha)d(y,\partial\Om)<r_0$, and so Theorem~\ref{cg-thm21} can be applied. Notice that since $A_{2(1+\alpha)d(y,\partial\Om)}(q_y)$ is a corkscrew point we know that
$$
d\left(A_{2(1+\alpha)d(y,\partial\Om)}(q_y),\partial\Om\right)\ge \frac{2(1+\alpha)d(y,\partial\Om)}{M}.
$$
Moreover, $d(A_{2(1+\alpha)d(y,\partial\Om)}(q_y),y)\le  4(1+\alpha)d(y,\partial\Om)$, as both $y$ and $A_{2(1+\alpha)d(y,\partial\Om)}(q_y)$ lie in a ball $B_d(q_y, 2(1+\alpha)d(y,\partial\Om))$. 

Set $\varepsilon:=\min\{{d(y,\partial\Om),\frac{2(1+\alpha)d(y,\partial\Om)}{M}}\}\approx d(y,\partial\Om)$. Then $d(A_{2(1+\alpha)d(y,\partial\Om)}(q_y),y)\le C\varepsilon$ with constant $C$ depending only on $\alpha$ and $M$ and independent of $y$. Therefore, there is a Harnack chain joining $y$ and $A_{2(1+\alpha)d(y,\partial\Om)}(q_y)$ with length independent of $y$ and hence by the Harnack inequality
\begin{equation}\label{thm13-aux4}
G(z,A_{2(1+\alpha)d(y,\partial\Om)}(q_y))\approx_{\alpha, M, C}G(z,y).
\end{equation}
Thus, by combining \eqref{thm13-aux3} and~\eqref{thm13-aux4} and applying them at~\eqref{thm13-aux1}, we obtain~\eqref{thm13-aux2}, as desired.
\smallskip

\noindent Finally, we apply \eqref{ident-prop513} to arrive at
\[
\int_{\partial \Om}\int_{\Gamma_{\alpha}^h(x)}|\nabla_H u(y)|^2d(y,\partial \Om)^{2-Q}\ud y \ud \om^z(x) 
\leq\frac12 \int_{\partial \Om} |f(y)-u(z)|^2\ud \om^z(y) \leq 2 \int_{\partial \Om} |f(y)|^2\ud \om^z(y).
\] 
Adding up together this estimate and~\eqref{est-non-trunc} we obtain the assertion of the theorem.
\end{proof}

\subsection{Proof of Theorem~\ref{thm-jk96}}% and its consequences

%\begin{proof}[Proof of Theorem~\ref{thm-jk96}]
 The structure of the proof follows the corresponding one for the proof of Theorem 9.6 in~\cite{jk}. However, the subriemannian setting of $\Hn$ requires applying different tools.
 
{\sc Claim 1.} \emph{ Let $x_0\in \partial \Om$, $r<r_0$ and $A_r(x_0)\in \Om$ be an internal corkscrew point as in Definition~\ref{def:NTA}(1) such that it satisfies $\frac{r}{M} <d(A_r(x_0), x_0)<Mr$. Then it holds that
\begin{equation}\label{claim1-thm96}
v(A_r(x_0)):=\int_{\partial \Om\setminus \Delta_1} |f(x)-f_{\Delta_1}|\ud \om^{A_r(x_0)}(x) \leq C \|f\|_{BMO(\partial \Om),}
\end{equation}
where $C$ is independent of $r$ and $f$ and $\Delta_1:=B(x_0,2r)\cap \partial \Om$ the surface ball.} 

%\komT{Czy tak chcemy zdefiniowac kule na brzegu? Ta definicja okresla rowniez definicje $BMO(\partial \Om)$, zatem jest istotna, cf. Def. 8.4 w \cite{jk}.}

In order to the prove the claim, we repeat the reasoning from the proof of Lemma 9.7 in~\cite{jk}, see also the proof of Lemma on pg. 35 in~\cite{fn}.

 Fix $y\in \Om$ and consider the harmonic measure $\om^{y}$.  Define $\Delta_j:=B(x_0,2^jr)\cap \partial \Om$ and the related ring domains $R_j:=\Delta_j\setminus \Delta_{j-1}$ for $j=1,2,\ldots$. Recall the notation $f_{\Delta_j}:=\vint_{\Delta_j} f\ud \om^y$.
 
 By virtue of Theorem 11 and Proposition 6 in~\cite{cg98} we have that for kernel functions $K$ it holds 
 \begin{align*}
  v(A_r(x_0)) &\leq \sum_{j\geq 2} \int_{R_j}|f(x)-f_{\Delta_j}|\ud\om^{A_r(x_0)}(x)+\sum_{j\geq 2} \int_{R_j}|f_{\Delta_1}-f_{\Delta_j}|\ud\om^{A_r(x_0)}(x)\\
  & \leq \sum_{j\geq 2} \int_{R_j}|f(x)-f_{\Delta_j}|K(A_r(x_0), x)\ud\om^{y}(x)+\sum_{j\geq 2} |f_{\Delta_1}-f_{\Delta_j}|\int_{R_j}K(A_r(x_0), x)\ud\om^{y}(x)\\
  &\leq \sum_{j\geq 2} \frac{C2^{-\kappa j}}{\om^{y}(\Delta_j)}\int_{\Delta_j}|f(x)-f_{\Delta_j}|\ud\om^{y}(x)+\Big(\sup_{k\geq 1} |f_{\Delta_k}-f_{\Delta_{k-1}}|\Big) \sum_{j\geq 2} jC2^{-\kappa j}\frac{\om^{y}(R_j)}{\om^{y}(\Delta_j)}.
 \end{align*}
Note that here $\kappa>0$ denotes the constant $\alpha>0$ in Proposition 6 in~\cite{cg98}.
Finally, the standard argument involving mean value integrals gives us that
\begin{align*}
 \left|f_{\Delta_k}-f_{\Delta_{k-1}}\right| &= \frac{\om^y(\Delta_k)}{\om^y(\Delta_{k-1})} \left(\frac{1}{\om^y(\Delta_k)}\int_{\Delta_k}|f(x) -f_{\Delta_{k}}|\ud\om^y(x)\right) \\
 &\leq C\left(\frac{1}{\om^y(\Delta_k)}\int_{\Delta_k}|f(x) -f_{\Delta_{k}}|\ud\om^y(x)\right).
\end{align*}
Here we also appeal to the doubling property of $\om^y$, see Theorem 2 in~\cite{cg98} which is proven for $k$ small enough so that radii $2^kr<r_0$. In order to obtain this property for large $k$ we use the Harnack inequality and Corollary 3 in~\cite{cg98}, cf. the discussion following Lemma 4.9 in \cite{jk} and the proof of Lemma 4.2 therein.

% \komT{Argumentacja w \cite{jk} nie jest dla mnie do konca jasna: czy aproksymuja kawalkami stala funkcje przez funkcje ciagle? Ponadto, Cor. 3 w\cite{cg98} jest troche slabsze niz lemma 2 w \cite{jk}. }\komM{Nie rozumiem tej uwagi, porozmawiamy o niej, jak sie spotkamy.}
%\komT{Tu nie jest dla mnie jasne co jer-kenig maja na mysli w definicji 8.4 brzegowego BMO: ustalaja $y$, czy jednak mysla o tej przestrzeni jako rownowaznej z $BMO(\ud \sigma)$ dla miary powierzchniowej?}

By applying the definition of the seminorm in $BMO(\partial \Om)$, cf. Definition~\ref{defn-bmo-bd}, we obtain 
\[
 v(A_r(x_0)) \leq \|f\|_{BMO(\partial \Om)} \Big(C+C\sum_{j\geq 2}j2^{-\kappa j}\Big)\leq C \|f\|_{BMO(\partial \Om)}
\]
and Claim 1 is proven.

{\sc Claim 2.} \emph{Let $x_0\in \partial \Om$ and $A_r(x_0)\in \Om$ be an internal corkscrew point as in Definition~\ref{def:NTA}(1). Denote by $\Delta(x_0,r):=B(x_0,r)\cap \partial \Om$ the surface ball. Then it holds that
\begin{equation}\label{claim-l98}
\frac{r^{2(Q-2)}}{(\om^{A_r(x_0)}(\Delta(x_0,r)))^3}\int_{B(x_0,r)\cap \Om}\frac{G^3(x, A_r(x_0)}{d(x,\partial \Om)^2}\ud x \leq C,
\end{equation}
where $C$ depends only on $n$, the geometry of $\Hn$ and $r_0$ and $M$ (the NTA parameters of $\Om$).
}

We again follow the corresponding proof of Lemma 9.8 in~\cite{jk}, although observe that instead of the dyadic Whitney cubes we need a different family of sets covering $B\cap \Om$. Such family can be constructed by the direct modification of the proof of Proposition 4.1.15 in~\cite{hkst} as follows:

\emph{
There exists a countable family of balls in $B(x_0, r)\cap \Om$ denoted by 
\[
\mathcal{F}:=\left\{B_d(x_i, \frac{1}{8}d(x_i, \partial \Om))\right\},\quad x_i\in \Om,
\]
such that each ball $B_d(x_i)$ has a non-empty intersection with set $B(x_0,r)\cap \Om$ and, moreover, $\sum_{i}\chi_{2B_d(x_i)}\leq 2N^5$, where $N$ stands for the doubling constant in $\Hn$.
}

Let us define the following subfamily of $\mathcal{F}$:
\[
 \mathcal{F}_k:=\left \{B_d(x_i)\in \mathcal{F}\,:\,2^{-k}\leq \frac{1}{8}d(x_i, \partial \Om) \leq 2^{-k+1}\right\},\quad k=-\lceil \log_2 r_0\rceil-1,\dots,0,1,2,\ldots.
\]
Notice that we do not need to take exponents bigger than $\lceil \log_2 r_0\rceil+1$ because any $x_i$ has to satisfy 
\[
d(x_i, \partial \Om)\le d(x_i,x_0)\le \frac{8}{7}r\le\frac{8}{7}r_0.
\]
In order to prove the second inequality above let us set $d(x_i,x_0)=(1+t)r$, for some $t\in\mathbb{R}$ and notice that $d(x_i,x_0)\le r+\frac{1}{8}d(x_i,\partial\Om)$.Thus, we have
\[
(1+t)r=d(x_i,x_0)\le r+\frac{1}{8}d(x_i,\partial\Om)\le r+\frac{1}{8}d(x_i,x_0)=r+\frac{1}{8}(1+t)r,
\]
and so $t\le\frac{1}{7}$.

With the above introduced notation we may reduce the estimate in~\eqref{claim-l98} to the estimate over the balls in $\mathcal{F}$:
\begin{equation}\label{est-thmjk-0}
\int_{B(x_0,r)\cap \Om} \frac{G^3(x, A_r(x_0)}{d(x,\partial \Om)^2}\ud x\leq 2N^5\sum_{k}\sum_{B_d(x_i)\in \mathcal{F}_k} \int_{B_d(x_i)} \frac{G^3(x, A_r(x_0)}{d(x,\partial \Om)^2}\ud x.
\end{equation}
For a fixed $k$ and given ball $B_d(x_i)\in \mathcal{F}_k$ let $x_i^*\in \partial \Om$ denote a point such that
$d(x_i,x_i^*)=d(x_i,\partial \Om)$. That such a point exists is a consequence of compactness of $\overline{\Om}$, as $\Om$ is a bounded domain.

Set $\Delta_i:=\Delta(x_i^*, 2^{-k+1})=B(x_i^*, 2^{-k+1})\cap \partial \Om$, a surface ball. By the Harnack inequality for a harmonic function $G(\cdot,A_r(x_0))$ applied on a ball $B_d(x_i)$ we have that $G(x,A_r(x_0))\approx_{C}G(x_i, A_r(x_0))$ and, thus,
\begin{equation}\label{est-thmjk-1}
 \int_{B_d(x_i)\in \mathcal{F}_k} \frac{G^3(x, A_r(x_0))}{d(x,\partial \Om)^2}\ud x\approx G^3(x_i, A_r(x_0)) 2^{2k} \left(\frac{1}{8}d(x_i,\partial \Om)\right)^Q\approx G^3(x_i, A_r(x_0)) 2^{k(2-Q)}.
\end{equation}
Indeed, if $x\in B_d(x_i)\in \mathcal{F}_k$, then $d(x_i,\partial \Om) \lesssim d(x,\partial \Om) \leq d(x_i,\partial \Om)+d(x,x_i)\leq 16\cdot2^{-k+1}$ and so $d(x,\partial \Om)\approx 2^{-k}$. 

Next we show that we may consider points $x_i$ as the corkscrew points in Definition~\ref{def:NTA}(1), so that $x_i:=A_{2^{-k+4}}(x_i^*)$. Since $B_d(x_i)\in\mathcal{F}_k$ we have
\[
d(x_i,x^*_i)=d(x_i,\partial\Om)\le 2^{-k+4}.
\]
 On the other hand
\[ 
d(x_i,x^*_i)=d(x_i,\partial\Om)\ge \frac{2^{-k+4}}{2}\ge\frac{2^{-k+4}}{M}
\]
for any $M\ge 2$. However, in the definition of NTA domains we only have existence of some constant $M$. If it happens that that $M<2$, we can always increase it without losing anything in said definition. Hence, we can assume $M\ge 2$. This shows that indeed $x_i=A_{2^{-k+4}}(x_i^*)$. Now let us choose a point $y_0\in \Om\setminus B_d(x_i^*,a2^{-k+4})$. In fact, one can take $a=M$, see Section 5 \cite{cgn} and, moreover, assume that   $y_0:=A_{\widetilde{C}r}(x_i^*)$ with $\widetilde{C}=\widetilde{C}(M)$.

We apply Theorem~\ref{cg-thm21} to obtain the following estimate
\begin{align}
G(x_i, A_{\widetilde{C}r}(x_i^*))&=G(A_{\widetilde{C}r}(x_i^*),A_{2^{-k+4}}(x_i^*)) \nonumber \\
&\approx \frac{2^{2(-k+4)}}{|B_d(A_{\widetilde{C}r}(x_i^*), 2^{-k+4})|} \om^{A_{\widetilde{C}r}(x_i^*)}(\Delta(x_i^*,2^{-k+4}))\nonumber \\
&\lesssim 2^{(-k+4)(2-Q)} \om^{A_{\widetilde{C}r}(x_i^*)}(\Delta_i), \label{claim2-aux1}
\end{align}
where in the last inequality we use the doubling property of the harmonic measure. Furthermore, observe that
\begin{equation}\label{claim2-aux2}
G(x_i, A_{\widetilde{C}r}(x_i^*))=G(A_{\widetilde{C}r}(x_i^*), x_i)\approx_{C,M} G(A_r(x_0), x_i)=G(x_i, A_r(x_0)),
\end{equation}
and the change of point $A_{\widetilde{C}r}(x_i^*)$ to $A_r(x_0)$ in the middle estimate requires explanation:  Indeed, choose $A_{\widetilde{C}r}(x_i^*)\in\Om\setminus B_d(x_i^*,a2^{-k+4})$ and observe that, by the above discussion, we can consider the same $y_0=A_{\widetilde{C}r}(x_i^*)$ for any point $x_i^*$. We also recall that $d(y_0,\partial\Om)\ge \frac{\widetilde{C}}{M}r$ and that by the definition $d(A_r(x_0),\partial\Om)\ge\frac{r}{M}$. Moreover, we can assume that $d(A_r(x_0),A_{\widetilde{C}r}(x_i^*))\le Cr$, where $C=C(M)$. Therefore, points $A_r(x_0)$ and $A_{\widetilde{C}r}(x_i^*)$ can be joined by a Harnack chain of length depending only on $C(M)$. This observation together with the Harnack inequality allow us to replace in~\eqref{claim2-aux2} point $A_{\widetilde{C}r}(x_i^*)$ with $A_r(x_0)$ by the price of possibly increasing constants. 

Upon applying estimates~\eqref{claim2-aux1} and~\eqref{claim2-aux2} in~\eqref{est-thmjk-1} we obtain the following:
\begin{equation}\label{est-thmjk-2}
 \int_{B_d(x_i)\in \mathcal{F}_k} \frac{G^3(x, A_r(x_0))}{d(x,\partial \Om)^2}\ud x\approx G^2(x_i, A_r(x_0))\om^{A_{\widetilde{C}r}(x_i^*)}(\Delta_i).
\end{equation}
In order to estimate the expression on the right-hand side, we appeal to the Carleson-type estimate, see Lemma 1 in~\cite{cg98}. Recall that $G\geq 0$ in $\Om$ and $G(\cdot, A_r(x_0))\equiv 0$ on $\partial \Om$. Moreover, notice that $\Delta(x_0, 2r)\subset \Delta(x_i^*, Cr)$. Indeed, let $y\in \Delta(x_0, 2r)$. Then,
\[
 d(y, x_i^*)\leq d(y, x_0)+d(x_0, x_i)+d(x_i, x_i^*) \leq r + (r+r2^{-k+1})+2^{-k+1}\lesssim Cr.
\]
The last step requires that $2^{-k}\leq Cr$ and under our assumptions this restriction is sufficient. Since, otherwise suppose that $2^{-k}\geq r$. Then for any $x_i$ it holds that $d(x_i, \partial \Om)\geq 8\cdot 2^{-k}\geq 8r$. However, on the other hand,
$$
d(x_i,\partial \Om)\leq d(x_i,x_0)\leq r+2r_i \leq r+\frac{1}{4}d(x_i,\partial \Om).
$$
Hence $d(x_i,\partial \Om)\leq \frac43 r$, giving the contradiction.

We apply Lemma 1 in \cite{cg98} on $B_d(x_i^*,Cr)$ to get that, for an exponent $\beta>0$, the following holds
 \begin{equation}\label{est-thmjk-a}
G^2(x_i, A_r(x_0)) \leq C(M, r_0) \left(\frac{d(x_i, x_i^*)}{Cr}\right)^{2\beta}\Big(\sup_{x\in \partial B_d(x_i^*,Cr)\cap \Om} G(x, A_r(x_0))\Big)^2.
 %\lesssim_{C} \left(\frac{8\cdot 2^{-k+1}}{Cr}\right)^{2\beta} r^{2(2-Q)} (\om^{z}(\Delta(x_0,r)))^2, 
\end{equation}
Denote by $z \in \partial B_d(x_i^*,Cr)\cap \Om$ a point, where function $G(\cdot, A_r(x_0))$ attains its maximum. (Notice that this maximum cannot be obtained at a point in $\partial B_d(x_i^*,Cr)\cap \partial \Om$, as then it would be zero, as $G(\cdot, A_r(x_0))\equiv 0$ on $\partial \Om$ and by the maximum principle $G$ would be zero on $B_d(x_i^*,Cr)\cap \Om$). Therefore, the Carleson estimate in Theorem 9 in \cite{cg98} gives us that
\[
G(z, A_r(x_0))\lesssim_{M, r_0} G(A_{Cr}(x_i^*), A_r(x_0)).
\]
Here, in order to apply~\cite[ Theorem 9]{cg98}, we need to slightly increase constant $C$ on the right-hand side of the estimate, so that point $z$ belongs to $B_d(x_i^*,Cr)\cap \Om$. Moreover, in the case $A_r(x_0)\in B_d(x_i^*,2Cr)\cap \Om$ one needs additional chaining argument, by the definition of the NTA domains, to join $A_r(x_0)$ with the point $A_{Cr}(x_0)\not \in B_d(x_i^*,2Cr)\cap \Om$. This however, can be done by the price of increasing again constant $C$. From this discussion and by~\eqref{est-thmjk-a} we infer, by the Dahlberg-type estimate in Theorem~\ref{cg-thm21}, that
\[
G^2(x_i, A_r(x_0))\lesssim_{C} \left(\frac{8\cdot 2^{-k+1}}{Cr}\right)^{2\beta} r^{2(2-Q)} (\om^{A_{Cr}(x_i^*)}(\Delta(x_0,r)))^2.
\]
Then we apply the last estimate in~\eqref{est-thmjk-2} and in~\eqref{est-thmjk-0} to arrive at the following inequality
\[
\int_{B(x_0,r)\cap \Om} \frac{G^3(x, A_r(x_0)}{d(x,\partial \Om)^2}\ud x\lesssim_{C} 2N^5\sum_{k}
\left(\frac{2^{-k}}{r}\right)^{2\beta} r^{2(2-Q)} (\om^{A_{Cr}(x_i^*)}(\Delta(x_0,r)))^2\om^{A_{\widetilde{C}r}(x_i^*)}(\Delta_i).
\]
Finally, recall that by the discussion following~\eqref{claim2-aux2} we may join points $A_{\widetilde{C}r}(x_i^*)$ and $A_r(x_0)$ by the Harnack chain whose length depends only on $M$. By applying this observation, we conclude that 
\[
 \om^{A_r(x_0)}(\Delta(x_0,r))\approx_{C, M} \om^{A_{Cr}(x_i^*)}(\Delta(x_0,r)).
\]
Hence, \eqref{est-thmjk-0} becomes
\begin{equation*}
\int_{B(x_0,r)\cap \Om} \frac{G^3(x, A_r(x_0)}{d(x,\partial \Om)^2}\ud x\leq 2N^5\sum_{k} \left(\frac{2^{-k}}{r}\right)^{2\beta} r^{2(2-Q)} (\om^{A_{r}(x_0)}(\Delta(x_0,r)))^3
\end{equation*}
and thus the proof of Claim 2 is completed.
%\komT{Ad 6.8 i 6.10: gdyby pominac ADP to w ostatnim oszacowaniu pojawi sie $2^k\cdot 2^k$ , ktore z potega $2^{-2k\beta}$ daje skonczona sume?
%
%$$
%\int_{B(x_0,r)\cap \Om} \frac{G^3(x, A_r(x_0)}{d(x,\partial \Om)^2}\ud x\leq 2N^5\sum_{k} 
%\left(\frac{2^{k(\frac{1}{\beta}-1)}}{r}\right)^{2\beta} r^{2(2-Q)} (\om^{A_{r}(x_0)}(\Delta(x_0,r)))^3.
%$$
%Niestety, $\beta<1$.
%}
\smallskip

{\sc Continuation of the proof of Theorem~\ref{thm-jk96}.}
\smallskip

We are now in a position to complete the proof of Theorem~\ref{thm-jk96}. Suppose that $f\in BMO(\partial \Om, \ud \om)$ and $u$ is the harmonic function in $\Om$ such that $u(x)=\int_{\partial \Om} f(y)\ud \om^x(y)$.

Let $B(x_0,r)$ be a ball centered at $x_0\in \partial \Om$ with radius $r<\min\{1, r_0\}$ and denote by $\Delta_1=2\Delta:=B(x_0,2r)\cap \partial \Om$. % and by $\Delta:=B(x_0,r)\cap\Om. 
We modify the boundary data as follows:
\[
 f_1:=(f-f_{\Delta_1})\chi_{c\Delta_1},\qquad  f_2:=(f-f_{\Delta_1})\chi_{\partial \Om \setminus c\Delta_1}
\]
and, as in \cite{jk} we let $u_1$ and $u_2$ be their harmonic extensions, respectively, i.e.
\[
 u_i(x)=\int_{\partial \Om} f_i(y)\ud \om^x(y),\quad i=1,2.
\]
By direct application of Proposition~\ref{prop-thm514} to $f_1$ and $u_1$ we obtain that
\begin{align}
 &\frac{1}{\om^{A_{r}(x_0)}(\Delta)} \int_{B(x_0,r)\cap \Om} |\nabla_H u_1|^2 G(x,A_{r}(x_0))\ud x \nonumber \\
 &\phantom{AAAAA}\leq  \frac{1}{\om(\Delta)} \int_{\Om} |\nabla_H u_1|^2 G(x, A_{r}(x_0))\ud x \label{thm-jk-est-u1}\\
 &\phantom{AAAAA}=\frac12 \frac{1}{\om(\Delta)} \int_{\partial \Om} |(f(y)-f_{\Delta_1})\chi_{c\Delta_1}|^2\ud \om^z(y)
 <\|f\|^2_{BMO(\partial \Om, \ud \om)}, \label{thm-jk-est-u2}
\end{align}
where in the last inequality we use the John--Nirenberg theorem to get the equivalent definition of the BMO spaces in terms of the $L^2(\partial \Om, \ud \om)$-functions with $L^2$-integrable means, well-known in the Euclidean setting. Indeed, such an equivalent definition holds, as by Theorem 2 in~\cite{cg98}, the harmonic measure $\om^z$ is doubling in $\Om$ for $z$ enough away from the boundary of $\Om$ and we may repeat the appropriate part of the reasoning in Proposition 3.19~\cite{bb-book}, as long as the John--Nirenberg lemma holds for the surface balls $\Delta$. However, this follows by direct application of Theorem 5.2 in~\cite{aalto}:
\smallskip

\emph{
 Let $\Delta=B(x_0, r)\cap \partial \Om$ be a surface ball and $f\in BMO(\partial \Om, \ud \om^z)$. Then for all $\lambda>0$ 
 \[
  \om^z(\{z\in \Delta: |f(x)-f_{\Delta}| > \lambda \}) \leq c_1 \om^z(\Delta) e^{-\frac{c_2\lambda}{\|f\|_{BMO(\partial \Om, \ud \om^z)}}},
 \]
 where $c_1, c_2$ do not depend on $f$ and $\lambda$.
}
\smallskip
\\
The proof of John-Nirenberg lemma in \cite{aalto} requires only that the measure is doubling. By applying this result to the metric measure space $(\partial \Om, d|_{\partial \Om}, \ud \om^z)$, we conclude that indeed~\eqref{thm-jk-est-u2}  holds.

 By the gradient estimate for harmonic functions (see~\cite[Proposition 2.1]{lu}) and by the Harnack inequality on $B_d$-balls, we have that for any $x\in \Om$
\[
 |\nabla_H u_2(x)|\leq \frac{4c(n)}{d(x,\partial \Om)} \sup_{B_d(x,\frac14 d(x,\partial \Om))} |u_2|\leq \frac{C}{d(x, \partial \Om)} \int_{\partial \Om \setminus c\Delta_1}|f-f_{\Delta_1}|\ud \om^x. 
\]
We denote the last integral by $v(x)$ and apply the (local) boundary Harnack inequality in Theorem~\ref{cg-thm22} to $u_2$ and $G(\cdot, A_r(x_0))$, followed by the use of~\eqref{claim1-thm96} in Claim 1 and the Dahlberg-type estimate in Theorem~\ref{cg-thm21}, to arrive at the following estimate holding for $x\in B(x_0,r)\cap \Om$
\[
 v(x)\lesssim_{n, M, r_0} \frac{v(A_{2Mr}(x_0))}{G(A_{2Mr}(x_0),A_r(x_0))} G(x, A_r(x_0))\lesssim_{n, M, r_0} \|f\|_{BMO(\partial \Om, \ud \om)} \frac{r^{Q-2}}{\om^{A_{2Mr}(x_0)}(\Delta(x_0,r))} G(x, A_r(x_0)).
\]
Since $\Om$ in an NTA domain, we may apply the Harnack chain condition to join points $A_r(x_0)$ and $A_{2Mr}(x_0)$ with the chain of at most $CM$ balls and invoke the Harnack inequality to conclude that 
\[
\om^{A_{2Mr}(x_0)}(\Delta(x_0,r))\approx_{C, M} \om^{A_{r}(x_0)}(\Delta(x_0,r)).
\]
We apply this observation together with estimates for $v$ and $|\nabla_H u_2|$ and apply~\eqref{claim-l98} in Claim 2 to obtain that
\begin{align*}
  \frac{1}{\om^{A_{r}(x_0)}(\Delta)} \int_{B(x_0,r)\cap \Om}|\nabla_H u_2|^2 G(x, A_r(x_0))\ud x &\leq
   C \|f\|^2_{BMO(\partial \Om, \ud \om)} \frac{r^{2(Q-2)}}{(\om^{A_{r}(x_0)}(\Delta))^3}\int_{B(x_0,r)\cap \Om} \frac{G(x, A_r(x_0))^3}{d(x,\partial \Om)^2} \ud x \\
& \leq C \|f\|^2_{BMO(\partial \Om, \ud \om)}.
\end{align*}
Finally, we combine this estimate with the previous one for $|\nabla_h u_1|$, see~\eqref{thm-jk-est-u2}, and note that $\nabla_H u=\nabla_h u_1+\nabla_H u_2$. From this the assertion of Theorem~\ref{thm-jk96} follows.
\hfill $\square$

%We proceed to the proof of the second main result of this section, namely the $L^2(\ud \om)$-boundedness of the square function. 

\section{The Fatou-type theorem}

The goal of this section is to prove a version of the harmonic Fatou theorem in the Heisenberg setting. The studies of such theorems have led to several important notions and results to which our manuscript appeals to, for instance, the NTA domains, the area function and the non-tangential maximal function, see e.g.~\cite[Section 1]{jk}. The classical Fatou theorem asserts that a bounded harmonic function defined on the half-space in $\Rn$ has non-tangential limits at almost every point of the boundary, see e.g.~\cite[Theorem 2, Ch. VII]{st}, see also~\cite{car} for the local version. For the NTA domains in the Euclidean spaces, the Fatou theorem with respect to the harmonic measure is due to~\cite[Theorem 6.4]{jk}. In the subriemmanian setting the analogous results are proven in~\cite[Theorem 4]{cg98} for bounded NTA domains.
% and assert existence of non-tangential limits at almost all boundary points with respect to the harmonic measure.

We show a counterpart of the Fatou theorem for $(\ep,\delta)$-domains and, thus, for more general domains than the NTA ones, see the discussion below. Moreover, we are able to show the refinement of classical results, namely that nontangential limits of a harmonic function $u$ exist outside a set of $p$-capacity zero, not only zero measure. This, however, is obtained under stronger assumption on the global $L^p$-integrability of the gradient of harmonic function. 

We will now recall necessary definitions.
%\begin{defn}[Franchi \cite{fr}]\label{defn: John}
%A domain $\Om$ is  a John domain if there exists a point $x_0\in\Om$ and a constant $C>0$ such that for every point $x\in\Om$ there exists a rectifiable path parametrized by an arclength $\gamma:[0,T]\to\Om$ with $\gamma(0)=x_0$ and $\gamma(T)=x$ such that 
%\[
%d(\gamma(t),\partial\Om)\ge Ct.
%\]
%\end{defn}
%The definition of John domain was given by John in \cite{john}. One can also find it in the work by Martio and Sarvas in \cite{ms}. We use the definition given by Franchi in \cite{fr} as it was the one Franchi used to prove Theorem 2.31 in \cite{fr}.\komM{Moze lepiej nie pisac, bo recenzent bedzie chcial dyskusji o rownowaznosci definicji.}
%\komM{Skomentowac skad jest wzieta definicja dziedzin John. Pierwszy zdefiniowal je John, ogolniejsza definicje dali Martio, Sarvas}
\begin{defn}[cf. Definition 2.7 in~\cite{nh2}]\label{defn: epsilon-delta}
We say that a bounded domain $\Om\subset \Hn$ is an $(\varepsilon,\delta)$-domain if for all $x,y\in\Om$ such that $d(x,y)<\delta$ there exists a rectifiable curve $\gamma\subset\Om$ joining $x$ and $y$ satisfying
\[
l(\gamma)\le\frac{1}{\varepsilon}d(x,y),
\]
and
\[
d(z,\partial\Om)\ge \varepsilon \min\{d(x,z),d(y,z)\} \hspace{3mm} \textnormal{for all } z \textnormal{ on } \gamma.
\]
\end{defn}
 The definition of the $(\varepsilon,\delta)$-domains in the Euclidean setting was first given in~\cite{jones} and a fact that such domains are uniform and hence, John domains, is observed in Remark 4.2 in~\cite{va}. The above definition has also a counterpart in more general Carnot groups, see Definition 4.1 in~\cite{nh}, and leads to an extension theorem applied in the proof of Theorem~\ref{thm-Fatou} below, see Theorem 1.1 in~\cite{nh}. It is also known that a large class of NTA domains in $\Hn$ satisfies the definition of $(\varepsilon,\delta)$-domains, see Theorem 1.2 in~\cite{nh} and the discussion following it. Moreover, bounded $(\varepsilon,\delta)$-domains are uniform, see also~\cite{ct}.

For the definition and basic properties of $p$-Sobolev capacities we refer to \cite[Chapter 7.2]{hkst}.

\begin{theorem}\label{thm-Fatou}
 Let $\Om\subset \Hn$ be a bounded $(\ep, \delta)$-domain and let further $u$ be harmonic in $\Om$. If $\int_{\Om} |\nabla_H u|^p<\infty$ for some $1<p\leq 2n+2$, then $u$ has nontangential limits on $\partial \Om$ along horizontal curves in $\Om$ outside the set of $p$-Sobolev capacity zero.
\end{theorem}

The proof of the theorem employs among other results the following auxiliary observations. The proof of the first one is new in the literature, due to applying results of \cite{aw}.

\begin{lem}\label{lem-sub-mvp} Let $u$ be harmonic function in $\Om\subset \Hn$. Then, for any Kor\'anyi--Reimann ball $B_{KR}(x,r)\subset B_{KR}(x, 2r)\subset \Om$ and all $a\in \R$ we have that for any $p>1$ 
\[
 \sup_{B_{KR}(x,r)} |u(y)-a| \leq C(p,n)\left(\vint_{B_{KR}(x,2r)} |u(y)-a|^p \ud y\right)^{\frac{1}{p}}.
\]
\end{lem}
The result is well-known in the Euclidean setting and for $\mathcal{A}$-harmonic functions, see Lemma 3.4 in~\cite{hkm}.
\begin{proof}
 We apply the mean-value theorem in $\Hn$, see Theorem 4.4 in \cite{aw} (cf. Theorem 5.5.4 in \cite{blu}). By Definition 5.5.1 in \cite{blu} and pg. 253 we know that $|\nabla_H d|^2(x)\leq 1$ for any $x\not=0$ and so, we have that for any point $y\in B_{KR}(x,r)$
 \[
 |u(y)|\leq \vint_{B_{KR}(y,r)} |u(z)|\ud z.
 \]
By the H\"older inequality and the fact that if $u$ is harmonic then so is $u-c$, for any constant $c\in \R$, we obtain that
 \[
 |u(y)-a|\leq\left( \vint_{B_{KR}(y,r)} |u(z)-a|^p\ud z\right)^{\frac{1}{p}}. 
 \]
 Since for any $y\in B_{KR}(x,r)$ it holds that $B_{KR}(y,r)\subset B_{KR}(x,2r)$, the claim follows by the doubling property of the Lebesgue measure.
\end{proof}

\begin{lem}\label{lem-dens} Let $u\in HW^{1,p}(\Hn, \R)$ for some $1<p<2n+2$. Then
\[
 \lim_{r\to 0} \vint_{B(x,r)} |u(x)-u(y)|^p\ud x=0,
\]
for all points $x\in \Hn$ except for a set $E\subset \Hn$ of $p$-Sobolev capacity zero.
\end{lem}
The result is a counterpart of Lemma 3.2 in~\cite{kmv} and Theorem 3.10.2 in~\cite{ziem} proven for $\Rn$ and the Bessel capacity. The proof follows from more general results for complete metric measure spaces supporting the $p$-Poincar\'e inequality, see Theorem 4.5 in~\cite{kl} and also Theorem 9.2.8 in~\cite{hkst}.
% and Theorem 5.62 in~\cite{bb-book}.
The metric space $(\Hn, d_{CC}, \ud x)$ satisfies the assumption of these theorems, see e.g. the discussion on pg. 400-403 in~\cite{hkst}.

In the proof below we also need the following non-local version of the $p$-Poincar\'e inequality for a John domain $\Om\subset\Hn$, see Theorem 2.31 in~\cite{fr}: 
\begin{equation}\label{ineq-glob-pi}
 \int_{\Om}|f-f_{\Om}|^q \ud x\leq C_{\Om} \int_{\Om}|\nabla_H f|^p\ud x,
\end{equation}
where $1\leq p<Q$ and $Q=2n+2$, $q=\frac{pQ}{Q-p}=2+\frac{2}{n}$, and $f$ is a Lipschitz function. Moreover, the constant $C_{\Om}$ is independent of $f$. Here, we specialize the statement in~\cite{fr} to our setting. In particular, observe that the balance condition in \cite[Theorem 2.14]{fr} is with our $p,q$ and $Q$ equivalent to the so-called relative lower volume decay, cf. (9.1.14) in \cite{hkst}. This in turn holds if an underlying measure is doubling, which is the case for the Lebesgue measure in $\Hn$.

\begin{proof}[Proof of Theorem~\ref{thm-Fatou}]
 In the proof we follow the steps of the corresponding Euclidean result, cf. \cite[Theorem 3.1]{kmv}. Since $\int_{\Om} |\nabla_H u|^p<\infty$, it holds by the Poincar\'e inequality~\eqref{ineq-glob-pi} that $u\in HW^{1,p}(\Om)$, as $u$ is harmonic in $\Om$ and so analytic, in particular Lipschitz in a bounded domain $\Om$.
 
 We apply an extension result, see Theorem 1.1 in~\cite{nh} with $G=\Hn$ and $\mathcal{L}^{1,p}=HW^{1,p}$ allowing us to conclude that $u\in HW^{1,p}(\Hn)$ provided that $\Om$ is an $(\ep, \delta)$-domain. Notice that in the notation of~\cite{nh}, it holds that $0<{\rm rad}(\Om)<\diam \Om$, as $\Om$ is connected and bounded, cf. Definition 4.2 in~\cite{nh} and also~\cite{nh2}. 
 
Let us consider a cone $\Gamma_\alpha(x_0)$ at any $x_0\in \partial \Om\setminus E$, where $E$ is the set in Lemma~\ref{lem-dens}. Hence, for any $x\in \Gamma_\alpha(x_0)$ we have that
\[
 d(x,x_0)\leq (1+\alpha)d(x,\partial \Om).
\]
Therefore, it holds that
\[
 B\Big(x, \frac12 d(x,\partial \Om)\Big)\subset B\Big(x_0, \big(1+\alpha+\frac12\big) d(x,\partial \Om)\Big).
\]
Recall that the Kor\'anyi--Reimann distance and the subriemannian distance are equivalent in $\Hn$ with the constant depending on $n$, see Section 2.1, and thus we have that
\[
 B_{KR}\Big(x, \frac{c}{2} d(x,\partial \Om)\Big)\subset  B\Big(x, \frac12 d(x,\partial \Om)\Big)\subset B\Big(x_0, (1+\alpha+\frac12) d(x,\partial \Om)\Big) \subset B_{KR}\Big(x_0, \frac{1}{c}(1+\alpha+\frac12) d(x,\partial \Om)\Big).
\]
We apply Lemma~\ref{lem-sub-mvp} with $a=u(x_0)$ to get that 
\begin{align*}
 |u(x)-u(x_0)|&\leq C(p,n)\left(\vint_{B_{KR}(x,\frac12 d(x,\partial \Om))} |u(y)-u(x_0)|^p \ud y\right)^{\frac{1}{p}}\\
 &\leq C(p,n, c)\left(\vint_{B_{KR}(x_0, \frac{1}{c}(1+\alpha+\frac12) d(x,\partial \Om))} |u(y)-u(x_0)|^p \ud y\right)^{\frac{1}{p}},
\end{align*}
where in the last step we also use a consequence of the doubling property (the relative lower volume decay (9.1.14) in~\cite{hkst}):
\[
\frac{|B_{KR}(x_0, \frac{1}{c}(1+\alpha+\frac12) d(x,\partial \Om))|}{|B_{KR}(x,\frac12 d(x,\partial \Om))|}\lesssim_{n} \left(\frac{c}{3+2\alpha}\right)^{2n+2}. 
\]
%This can be further estimated by a constant depending on $n$ and $c$ only, by assuming that $x$ is close enough the boundary (say $d(x, \partial \Om)<1$). 
The assertion of the theorem now follows from Lemma~\ref{lem-dens} by letting $d(x,\partial \Om)\to 0$.
\end{proof}

\begin{appendix}
\section{The lower bound for a Green function}
The following result, applied in Example~\ref{ex-ball}, is of independent interest and to best of our knowledge did not yet appear in the literature on Green functions in the subriemmanian setting.
\begin{prop}[cf. (1.9) in Theorem (1.1), \cite{gw}] Let $\Om\subset \Hn$ be a domain and $G:\Om\times\Om\to \R$ be a Green function of $\Om$ associated with the Laplacian $\Delta_H$. Then, it holds 
\[
 G(z,y)\geq c(n, \Delta_H) d(z,y)^{2-Q},
\]
for all $z,y\in \Om$ satisfying $d(z,y)\leq \frac12 d(y,\partial \Om)$.
\end{prop}
\begin{proof}
 We follow the steps of the corresponding proof in~\cite{gw}. Recall that $G\geq 0$, $G(x,\cdot)=0$ for $x\in\partial \Om$, $G(x,y)=G(y,x)$ and, moreover, for any fixed $y\in \Om$, the following representation formula holds: $G(\cdot, y)=\Gamma(\cdot, y)-h_y(\cdot)$, where $\Gamma$ is the fundamental solution of $\Delta_H$ and $h_y$ is the harmonic function in $\Om$ having as boundary values the fundamental solution $\Gamma$ with the pole at $y\in \Om$ (in the PWB sense). Thus, $\Delta_H^x G(x,y)=-\delta_y(x)$ which in the weak sense reads:
 \begin{equation}\label{app-weak}
  \int_{\Om} \langle \nabla_H G(x,y), \nabla_H\phi(x) \rangle\, \ud x=\phi(y), \quad \hbox{for any }\phi \in C_{0}^{\infty}(\Om). 
 \end{equation}
  Let $z,y\in \Om$ satisfy the assumption $d(z,y)\leq \frac12 d(y,\partial \Om)$ and set $r:=d(z,y)$. Define the test function $\phi\in C_{0}^{\infty}(\Om)$ such that:
 $$
 0\leq \phi \leq 1 \hbox{ in }\Om,\quad \phi\equiv 1|_{B_d(y, \frac{r}{2})},\quad \phi\equiv 0|_{\Om\setminus B_d(y, r)} \,\hbox{ and also } |\nabla_H \phi|\leq \frac{C}{r}.
 $$
 
Then, by applying~\eqref{app-weak} with the above $\phi$ , we obtain that
\begin{equation}\label{app-eq2}
 1\leq \frac{C}{r} \int_{B_d(y, r) \setminus B_d(y, \frac{r}{2})} |\nabla_H G(x,y)|\,\ud x,
\end{equation}
for all $x\in B_d(y, \frac{r}{2})$. Similarly, we consider another test function $\eta(x):=G(x,y) \psi^2(x)$, where $\psi \in C_{0}^{\infty}(\Om)$ is such that 
$$
0\leq \psi \leq 1 \hbox{ in }\Om, \quad \psi\equiv 1|_{B_d(y, r)\setminus B_d(y, \frac{r}{2})},\quad \psi\equiv 0\, \hbox{ outside set } B_d(y, \frac{3r}{2}) \setminus B_d(y, \frac{r}{4}) \hbox{ and also }|\nabla_H \psi|\leq \frac{C}{r}.
$$
Since $\nabla_H \eta(x)=\nabla_H G(x,y)\psi^2(x)+2G(x,y)\psi(x) \nabla_H\psi(x)$, upon substituting this expression into~\eqref{app-weak}, we obtain the following equation:
\[
 \int_{\Om} |\nabla_H G(x,y)|^2\psi^2(x)\,\ud x+2\int_{\Om} \langle \nabla_H G(x,y), \nabla_H\psi(x) \rangle G(x,y)\psi(x)\, \ud x=0,
\]
as the support of $\psi$ does not contain $y$. Now, the standard inequality $2ab\leq \frac14 a^2+4b^2$ for $a,b\in \R$ together with the H\"older inequality imply that
\[
(1-\frac14)  \int_{B_d(y, r)\setminus B_d(y, \frac{r}{2})} |\nabla_H G(x,y)|^2\,\ud x \lesssim \frac{C^2}{r^2} \sup_{ B_d(y, \frac{3r}{2}) \setminus B_d(y, \frac{r}{4})} G^2(x,y)\, r^Q .
\]
We combine this estimate with~\eqref{app-eq2} to arrive at the following inequality
\begin{align*}
1 &\lesssim \frac{C}{r} \left(\int_{B_d(y, r) \setminus B_d(y, \frac{r}{2})} |\nabla_H G(x,y)|^2\,\ud x\right)^{\frac12} r^{\frac{Q}{2}} \\
 & \lesssim C r^{\frac{Q}{2}-1}\left(C^2 r^{Q-2}\,\sup_{B_d(y, \frac{3r}{2}) \setminus B_d(y, \frac{r}{4})} G^2(x,y)\right)^{\frac12} \\
& \lesssim_{n, \Delta_H} C^2r^{Q-2} G(z,y), 
\end{align*}
where in the last step we also appeal to the Harnack inequality for harmonic function $G(\cdot,y)$. Thus, the proof is completed upon recalling that $r=d(z,y)$.
\end{proof}
\end{appendix}

%\textit{Tomasz Adamowicz:} Institute of Mathematics, Polish Academy of Sciences, ul. \'Sniadeckich 8, Warsaw 00-656, Poland. E-mail address: tadamowi@impan.pl
%
%\textit{Marcin Grysz\'owka} 


\begin{thebibliography}{C-J-S}

%\bibitem[AG]{AG}
%K. Astala, and F. W. Gehring, \textit{Quasiconformal analogues of theorems of Koebe and
%Hardy-Littlewood}, Mich. Math. J. 32 (1985), 99-107.
%
%
%\bibitem[AIM]{AIM}
%
%K. Astala, T. Iwaniec, and G. Martin, \textit{Elliptic partial differential equations and quasiconformal mappings in the plane}, Princeton Mathematical Series,
%vol. 48, Princeton University Press, Princeton, NJ, 2009
\bibitem[AF]{af} T. Adamowicz and K. F\"assler, \emph{Hardy spaces and quasiconformal maps in the Heisenberg group}, J. Funct. Anal. 284 (2023), no. 6, Paper No. 109832. 

\bibitem[AG1]{AG1} T. Adamowicz and M. J. Gonz\'{a}lez, \emph{Hardy spaces for quasiregular mappings and composition operators}, J. Geom. Anal., 31(11), 11417--11427, 2021.

\bibitem[AG2]{AG2} T. Adamowicz and M. J. Gonz\'{a}lez, \emph{Hardy spaces and quasiregular mappings}, submitted, arxiv 2309.12947.

\bibitem[AW]{aw} T. Adamowicz and B. Warhurst, \emph{Mean value property and harmonicity on Carnot--Carath\'eodory groups}, Potential Anal. 52 (2020), no. 3, 497--525.

\bibitem[Ah]{ahl} L. Ahlfors, \emph{M\"obius transformations in several dimensions}, Ordway Professorship Lectures in Mathematics. University of Minnesota, School of Mathematics, Minneapolis, Minn., 1981.

%\bibitem[AIM]{aim} K. Astala, T. Iwaniec, G. Martin, \emph{Elliptic partial differential equations and quasiconformal mappings in the plane}, Princeton Mathematical Series, vol. 48, Princeton University Press, Princeton, NJ, 2009.

\bibitem[ABKY]{aalto} D. Aalto, L. Berkovits, O. E. Kansanen, H. Yue, \emph{John--Nirenberg lemmas for a doubling measure}, Studia Math. 204(4) (2011), 21--37. 

\bibitem[AK]{AK}
K. Astala, and P. Koskela,\textit{ $H^p$-theory for Quasiconformal Mappings}, Pure Appl. Math. Q. 7
(2011), no. 1, 19-50.

%\bibitem[\"A]{Ak} T. \"Akkinen, \emph{Radial limits of mappings of bounded and finite distortion.} J. Geom. Anal. 24 (2014), no. 3, 1298--1322.
%
\bibitem[BT]{bt} Z. Balogh, J. T. Tyson, \emph{Polar coordinates in {C}arnot groups}, Math. Z., 241(4) (2002), 697--730.

\bibitem[Be]{bel} A. Bella\"iche, The tangent space in sub-Riemannian geometry. - In: Sub-Riemannian geometry,
eds. A. Bella\"iche and J. J. Risler, Birkh\"auser, Basel, 1996, 1--78.

\bibitem[BB]{bb-book} A. Bj\"orn, J. Bj\"orn, \emph{Nonlinear Potential Theory on Metric Spaces}, EMS Tracts in Mathematics 17, European Math. Soc., Zurich.

%\bibitem[BI]{bi} B. Bojarski, T. Iwaniec, \emph{Analytical foundations of the theory of quasiconformal mappings in $R^n$}, Ann. Acad. Sci. Fenn. Ser. A I Math. 8(2) (1983), 257--324.


%\bibitem[BM]{BM} R. Ba\~nuelos and C. Moore, \emph{Probabilistic behavior of harmonic functions},
%Progress in Mathematics, 175. Birkh\"auser Verlag, Basel, 1999. xiv+204 pp.
%
\bibitem[BLU]{blu} A. Bonfiglioli, E. Lanconelli, F. Uguzzoni, \emph{Stratified Lie Groups and Potential Theory for Their Sub-laplacians}, Springer Monographs in Mathematics, Springer (2007). 

\bibitem[BH]{bh} S. Bortz, S. Hofmann,  \emph{Quantitative Fatou theorems and uniform rectifiability}, Potential Anal. 53(1) (2020), 329--355.

\bibitem[CG]{cg98} L. Capogna, N. Garofalo, \emph{Boundary behavior of nonnegative solutions of subelliptic equations in NTA domains for Carnot--Carath\'eodory metrics}, J. Fourier Anal. Appl. 4 (1998), no. 4--5, 403--432.

\bibitem[CGN]{cgn} L. Capogna, N. Garofalo, D.-M. Nhieu, \emph{Properties of harmonic measures in the Dirichlet problem for nilpotent Lie groups of Heisenberg type}, Amer. J. Math. 124 (2002), no. 2, 273--306.

\bibitem[CT]{ct} L. Capogna, P. Tang, \emph{Uniform domains and quasiconformal mappings in the Heisenberg group}. Man. Math., 86(3) (1995), 267--282.

\bibitem[Car]{car} L. Carleson, \emph{On the existence of boundary values for harmonic functions in several
variables}, Ark. Mat. 4 (1962), 393--399.

%\bibitem[D]{D} P. L. Duren, \textit {Theory of $H^p$ Spaces}, Academic Press (New York), 1970.

%\bibitem[FGW]{FGW}
%X. Fang, K. Guo and Z. Wang, \textit{Composition Operators on the Bergman Space via
%Quasiconformal Mapping}s, preprint, 2018, arXiv:1804.5352.

\bibitem[EG]{EG}  L. C. Evans, R. Gariepy, \textit{Measure theory and fine properties of functions. Studies in Advanced Mathematics.} CRC Press, Boca Raton, FL, 1992. viii+268 pp.

\bibitem[FB]{fn} E. B. Fabes, U. Neri, \emph{Dirichlet problem in Lipschitz domains with BMO data}, Proc. Amer. Math. Soc. 78 (1) (1980), 33--39.

\bibitem[Fr]{fr} B. Franchi, \emph{BV spaces and rectifiability for Carnot--Carath\'eodory metrics: an introduction},
Czechoslovak Academy of Sciences, Mathematical Institute, Prague, 2003, 72--132.

%\bibitem[HL]{hl} G.H. Hardy, J.E. Littlewood, \emph{A maximal theorem with function-theoretic applications},
%Acta Math. 54 (1930), no.1, 81--116.

\bibitem[G]{ga} J. B. Garnett,  \textit{Bounded analytic functions.} Academic Press (New York), 1981.


\bibitem[GMT]{gmt} J. Garnett, M. Mourgoglou, X. Tolsa, \emph{Uniform rectifiability from Carleson measure estimates and $\varepsilon$-approximability of bounded harmonic functions.}, Duke Math. J. 167(8) (2018), 1473--1524.


\bibitem[GP]{gp} N. Garofalo, N.-C. Phuc, \emph{Boundary behavior of $p$-harmonic functions in the Heisenberg group}, Math. Ann. 351 (2011), no. 3, 587--632.

%\bibitem[GRS]{grs} N. Garofalo, M. Ruzhansky, D. Suragan, \emph{On Green functions for Dirichlet sub-Laplacians on H-type groups}, J. Math. Anal. Appl.452(2017), no.2, 896--905.

\bibitem[GW]{gw} M. Gr\"uter, K.-O. Widman, \emph{The Green function for uniformly elliptic equations},
Manuscripta Math. 37(3) (1982), 303--342. 

\bibitem[Gr]{gr}  M. Grysz\'owka, \emph{$\ep$-Approximability and Quantitative Fatou Theorem on Riemannian Manifolds}, arXiv:2309.11264, submitted.

\bibitem[HKM]{hkm} J. Heinonen, T. Kilpel\"ainen, O. Martio, \emph{Nonlinear Potential Theory of Degenerate Elliptic Equations}, Dover Publications, Inc., 2006.

\bibitem[HKST]{hkst} J. Heinonen, P. Koskela, N. Shanmugalingam, J. Tyson, \emph{Sobolev spaces on metric measure spaces. An approach based on upper gradients},New Mathematical Monographs, 27. Cambridge University Press, Cambridge, 2015.

\bibitem[HLM]{hlm} S. Hofmann, P. Le, A. Morris, \emph{Carleson measure estimates and the Dirichlet problem for degenerate elliptic equations}, Anal. PDE 12(8)  (2019), 2095--2146.

\bibitem[HMM]{hmm} S. Hofmann, J. Martell, S. Mayboroda, \emph{Uniform rectifiability, Carleson measure estimates, and approximation of harmonic functions}, Duke Math. J. 165(12) (2016), 2331--2389.

\bibitem[HMMTZ]{hmmtz} S. Hofmann, J. Martell, S. Mayboroda, T. Toro, Z. Zhao, \emph{Uniform rectifiability and elliptic operators satisfying a Carleson measure condition}, Geom. Funct. Anal. 31(2) (2021), 325--401.


\bibitem[HMMM]{hmmm} S. Hofmann, D. Mitrea, M. Mitrea, A.J. Morris, \emph{$L^p$-square function estimates on spaces of homogeneous type and on uniformly rectifiable sets}, Mem. Amer. Math. Soc. 245 (2017), no. 1159, v+108 pp.

\bibitem[HT]{ht} S. Hofmann, O. Tapiola, \emph{Uniform rectifiability implies Varopoulos extensions}, Adv. Math. 390 (2021), Paper No. 107961, 53 pp.

%\bibitem[IN]{in} {T. Iwaniec, C. A. Nolder}, \emph{Hardy--Littlewood inequality for quasiregular mappings in certain domains in $R^n$}, Ann. Acad. Sci. Fenn. Ser. A I Math. 10 (1985), 267--282.

\bibitem[JK]{jk} D. Jerison, C. Kenig, \emph{Boundary behavior of harmonic functions in nontangentially accessible domains}, Adv. Math. 46(1) (1982), 80--147.

\bibitem[Joh]{john} F. John, \emph{ Rotation and strain}, Comm. Pure Appl. Math. 14 (1961), 391--413.

\bibitem[Jon]{jones} P. Jones, \emph{Quasiconformal mappings and extendability of functions in Sobolev spaces}, Acta Math. 147 (1-2) (1981), 71--88.

\bibitem[KL]{kl} J. Kinnunen, V. Latvala, \emph{Lebesgue points for Sobolev functions on metric spaces},
Rev. Mat. Iberoamericana 18(3) (2002), 685--700.


\bibitem[KR1]{kr} A. Kor\'anyi, H. M. Reimann, \emph{Horizontal normal vectors and conformal capacity of spherical rings
  in the Heisenberg group}, Bull. Sci. Math., II. S\'er., 111 (1987), 3--21.

\bibitem[KR2]{kr2} A. Kor\'anyi, H. M. Reimann, \emph{Foundations for the theory of quasiconformal mappings on the Heisenberg group}, Adv. Math. 111(1) (1995), 1--87.

%\bibitem[Gi]{Gi} D. Girela, \emph{Mean growth of the derivative of certain classes of analytic functions}, Math. Proc. Cambridge Philos. Soc. 112(2) (1992), 335--342.

%\bibitem[JW]{JW}
%D. Jerison, and A. Weitsman, \textit{On the means of quasiregular and quasiconformal mappings}, Proc. Amer. Math. Soc. 83 (1981), 304--306.
%
%\bibitem[J]{j} P. Jones, \emph{Extension theorems for BMO}, Indiana Univ. Math. J. 29 (1980), no. 1, 41--66.

\bibitem[KMV]{kmv} P. Koskela, J. Manfredi, E. Villamor, \emph{Regularity theory and traces of $\A$-harmonic functions.}
Trans. Amer. Math. Soc. 348 (1996), no. 2, 755--766.

%\bibitem[KM]{km} P. Koskela, V. Manojlovi\'c, \emph{Quasi-nearly subharmonic functions and quasiconformal mappings.} Potential Anal. 37 (2012), no. 2, 187--196.

\bibitem[LU]{lu} E. Lanconelli, F. Uguzzoni, \emph{On the Poisson kernel for the Kohn Laplacian}, Rend. Mat. Appl. (7) 17 (1997), no. 4, 659--677.

\bibitem[MS]{ms} O. Martio, J. Sarvas, \emph{Injectivity theorems in plane and space}, Ann. Acad. Sci. Fenn. Ser. A I Math. 4 (1979), 383-401.

\bibitem[MMMS]{mmms} D. Mitrea, I. Mitrea, M. Mitrea, B. Schmutzler, \emph{Calder\'on--Zygmund theory for second-order elliptic systems on Riemannian manifolds}, Integral methods in science and engineering, 413--426.
Birkh\"auser/Springer, Cham, 2015.
%\bibitem[LV]{LV} O. Lehto, and K. I. Virtanen, \emph{Quasiconformal mappings in the plane.} Second edition. Springer-Verlag, New York-Heidelberg, 1973. 

%\bibitem[LP]{LP}

%J. E. Littlewood and R. E. Paley, \emph{Theorems on Fourier series and power series (II)}. Proc. London Math. Soc. (2), 42 (1937), 52--89.


%\bibitem[L]{L}
%D. H. Luecking,\textit{ Forward and reverse Carleson inequalities for functions in Bergman spaces
%and their derivatives}, Amer. J. Math. 107 (1985), 85--111.

%\bibitem[Ma]{man} J.J. Manfredi, \emph{$p$-harmonic functions in the plane} Proc. Amer. Math. Soc. 103 (1988), no. 2, 473--479.
%
%\bibitem[Mi]{M1} R. Miniowitz, \emph{Distortion theorems for quasiregular mappings} Ann. Acad. Sci. Fenn. Ser. A I Math. 4 (1979), no. 1, 63--74.
%

\bibitem[Nh1]{nh2} D.-M. Nhieu, \emph{Extension of Sobolev spaces on the Heisenberg group}, C. R. Acad. Sci. Paris S\'er. I Math. 321(12) (1995), 1559--1564.


\bibitem[Nh2]{nh} D.-M. Nhieu, \emph{The Neumann problem for sub-Laplacians on Carnot groups and the extension theorem for Sobolev spaces}, Ann. Mat. Pura Appl. (4) 180(1) (2001), 1--25.

%\bibitem[No]{nol} C. Nolder, \emph{The $H^p$-norm of a quasiconformal mapping}, J. Math. Anal. Appl. 275 (2002), no.2, 557--561.
%
%\bibitem[N]{N} K. Noshiro, \emph{Cluster sets.} Ergebnisse der Mathematik und ihrer Grenzgebiete. N. F., Heft 28 Springer-Verlag, Berlin-G\"ottingen-Heidelberg 1960.


%\bibitem[Re]{resh} Y.G. Reshetnyak, \emph{Space mappings with bounded distortion} - Transl. Math. Monogr. 73,
%Amer. Math. Soc., Providence, RI, 1989.
%
%\bibitem[R]{Ric} S. Rickman, \emph{Quasiregular mappings}, Ergebnisse der Mathematik und ihrer Grenzgebiete (3) [Results in Mathematics and Related Areas (3)], 26. Springer-Verlag, Berlin, 1993. 
%

\bibitem[St]{st} E. M. Stein, \emph{Singular integrals and differentiability properties of functions}, Princeton Math. Ser., No. 30, Princeton University Press, Princeton, NJ, 1970, xiv+290 pp.

\bibitem[V\"a]{va}  J. V\"ais\"al\"a, \emph{Uniform domains}, Tohoku Math. J. 40 (1988), 101--118.

%\bibitem[V]{va}  J. V\"ais\"al\"a, \emph{Lectures on $n$-dimensional quasiconformal mappings} Lecture Notes in Mathematics, Vol. 229. Springer--Verlag, Berlin--New York, 1971. 

\bibitem[Zr]{ziem} W. Ziemer, \emph{Weakly differentiable functions}, Springer 1989.
\bibitem[Z]{z} M.  Zinsmeister, \emph{A distortion theorem for quasiconformal mappings} Bull. Soc. Math. France 114 (1986), no. 1, 123--133.

\end{thebibliography}
\end{document}